\documentclass[reqno,twoside,12pt,a4paper]{amsart}
\usepackage{layout}
\usepackage{latexsym}
\usepackage{amsmath}
\usepackage{amssymb}
\usepackage{cases}
\usepackage{ascmac}
\usepackage{mathrsfs}
\usepackage{amsthm}
\usepackage{tikz}
\usetikzlibrary{intersections, calc, arrows}
\newtheorem{theorem}{\textbf{Theorem}}[section]
\newtheorem{lemma}{\textbf{Lemma}}[section]
\newtheorem{proposition}{\textbf{Proposition}}[section]
\newtheorem{corollary}{\textbf{Corollary}}[section]
\newtheorem{remark}{\textbf{Remark}}[section]

\numberwithin{equation}{section}

\renewcommand{\theequation}{\thesection.\arabic{equation}}

\title[Fast diffusion equation with dynamic boundary conditions]
{On a perturbed fast diffusion equation with dynamic boundary conditions}

\author[T.\ Fukao]{Takeshi Fukao}
\address{Takeshi Fukao: Department of Mathematics, Faculty of Education, 
Kyoto University of Education, 
1~Fujinomori, Fukakusa, Fushimi-ku, Kyoto~612-8522 Japan}
\email{fukao@kyokyo-u.ac.jp}

\dedicatory{}
\subjclass[2000]{}
\begin{document}

\thispagestyle{empty}
\maketitle

\begin{abstract}
This paper discusses finite time extinction for 
a perturbed fast diffusion equation with dynamic boundary conditions. 
The fast diffusion equation has the characteristic property of decay, such as the solution 
decays to zero in a finite amount of time depending upon the initial data. 
In the target problem, some $p$-th or $q$-th order perturbation term may work to blow up within this period. 
The problem arises from the conflict between the diffusion and the blow up, in the bulk and on the boundary.  
Firstly, the local existence and uniqueness of the solution are obtained.  
Finally, a result of finite time extinction for some small initial data is presented.   
\smallskip


\end{abstract}

\section{Introduction}
\setcounter{equation}{0} 

In general, when discussing the well-posedness for some parabolic partial differential equations in a smooth bounded domain, 
the initial and boundary values are taken as auxiliary conditions. 
Settings with {D}irichlet, {N}eumann,  or {R}obin boundary conditions are common.  
Recently, dynamic boundary conditions have also been treated in several studies. Here, the boundary condition 
includes the time derivative. 
Moreover, the dynamic boundary condition with surface diffusion, which is the generalized {W}entzell 
(Ventcel') boundary condition \cite{Ven59}, is of great interest.  
The presence of a dynamic boundary condition in evolution problems creates
a transmission problem between the dynamics in the bulk 
and on the boundary. 
\smallskip

In this paper, we consider the fast diffusion equation with a dynamic boundary condition of the following form:
\begin{align}
	\partial _t u -\Delta u^m + a u^m = \lambda u^p 
	& \quad \mbox{in } \Omega, \quad t >0,
	\label{fast1}\\
	\partial _t u +\partial _{\boldsymbol{\nu}} u^m 
	-\Delta_\Gamma u ^m +b u ^m =  \mu u^q 
	& \quad \mbox{on } \Gamma, \quad t>0,
	\label{fast2}
\end{align}
where $\Omega \subset \mathbb{R}^3$ is a bounded domain with 
smooth boundary $\Gamma:=\partial \Omega$. 
We set up parameters $0<m \le 1$, $p, q > 1$, and $(a,b), (\lambda, \mu) \in \{ (1,0), (0,1)\}$.
The symbols $\partial _t $, $\Delta$, $\partial _{\boldsymbol{\nu}}$, and 
$\Delta_\Gamma$ denote the time derivative, the {L}aplacian, the normal derivative with respect to 
the outward unit normal vector $\boldsymbol{\nu}$ on $\Gamma$, and 
the {L}aplace--{B}eltrami operator (see, e.g., \cite{Gri09}), respectively. 
It is worth noting that $\Delta _\Gamma$ plays an important role in this paper. 
The second equation \eqref{fast2} is called the dynamic boundary condition. 
It describes the dynamics
on the boundary through the appearance of the time derivative. 
\smallskip

In general, if $a=\lambda=0$, then we can categorize the nonlinear parabolic equation \eqref{fast1} as a 
\emph{fast diffusion equation}, as compared with a 
\emph{heat equation} ($m=1$) and a \emph{porous medium equation} ($m>1$) (see, e.g., \cite{Vaz07}). 
The fast diffusion equation has the characteristic property of decay 
\cite{Sab62}. More precisely, the solution 
decays to zero in a finite time that depends upon the initial data. 
This is called \emph{finite time extinction}. 
In this paper, we consider the finite time extinction for the perturbed fast diffusion equation 
\eqref{fast1}--\eqref{fast2} with some initial condition. 
The $p$-th or $q$-th order perturbation term may work to the 
blow up with in this time. 
Several studies have been conducted on the 
perturbed fast diffusion equation with the homogeneous {D}irichlet boundary condition,  
\begin{align*}
	\partial _t u -\Delta u^m  = u^p 
	& \quad \mbox{in } \Omega, \quad t>0,
	\\
	u=0 
	& \quad \mbox{on } \Gamma, \quad t >0, 
	\\
	u(0)=u_0 
	& 
	\quad \mbox{in } \Omega,
\end{align*}
for example, \cite{Bal77, Fuj70} for $m=1$.  
For $0<m<1$, we refer to 
\cite{AK13, FF88, FF89, FF90, Filo87, Filo88, LW05, Vit00} and references therein.  
The behaviour of the solution differs from the case of $\Omega=\mathbb{R}^3$ (see \cite{Leo96, MM95}, for example). 
On the other hand, 
some studies related to \eqref{fast1} considered 
the nonlinear boundary condition \cite{CFQ91, Fila89, FK01, Filo92, LW07, ST20, Wol93}, 
and the dynamic boundary condition
\cite{FIK13, FIK15, FIK16, FIK17, FQ99, Gal12, GM14, Vit05}. 
The case of $m>1$ is also interesting (see \cite{FM92} for example, that considers equations similar to \eqref{fast1}--\eqref{fast2}), 
and has been the subject of several studies.
Many of which are related to the pioneering blow up results 
of \cite{Fuj66, Kap63}. 
\smallskip

To clarify the setting of our problem, we present the corresponding problems 
with parameter settings as follows:  
\begin{align}
	& (a,b,\lambda,\mu)=(1,0,1,0), 
	\quad 
	& 
	\begin{cases}
	\partial _t u -\Delta u^m + u^m = u^p
	& \quad \mbox{in } \Omega,
	\\
	\partial _t u +\partial _{\boldsymbol{\nu}} u^m 
	-\Delta_\Gamma u ^m =0 
	& \quad \mbox{on } \Gamma;
	\end{cases}
	\label{1010}
	\\
	& (a,b,\lambda,\mu)=(0,1,1, 0), 
	\quad 
	& 	
	\begin{cases}
	\partial _t u -\Delta u^m  = u^p  
	& \quad \mbox{in } \Omega,
	\\
	\partial _t u +\partial _{\boldsymbol{\nu}} u^m 
	-\Delta_\Gamma u ^m +u^m=0 
	& \quad \mbox{on } \Gamma;
	\end{cases}
	\label{0110}
	\\
	& (a,b,\lambda,\mu)=(1,0,0,1), 
	\quad 
	&	
	\begin{cases}
	\partial _t u -\Delta u^m + u^m =0 
	& \quad \mbox{in } \Omega,
	\\
	\partial _t u +\partial _{\boldsymbol{\nu}} u^m 
	-\Delta_\Gamma u ^m = u ^q  
	& \quad \mbox{on } \Gamma;
	\end{cases}
	\label{1001}
	\\
	& (a,b,\lambda,\mu)=(0,1,0,1), 
	\quad 
	& 
	\begin{cases}
	\partial _t u -\Delta u^m  =0 
	& \quad \mbox{in } \Omega,
	\\
	\partial _t u +\partial _{\boldsymbol{\nu}} u^m 
	-\Delta_\Gamma u^m +u^m = u^q 
	& \quad \mbox{on } \Gamma,
	\end{cases}
	\label{0101}
\end{align}
where the initial condition is omitted for simplicity. 
As a remark, we could also mention the case $(a,b)=(1,1)$.
However, we do not consider it in the present paper since it is trivial. 
\smallskip

In this paper, applying the method of {F}ilo \cite{Filo87}, 
we discuss the local existence and uniqueness of the non-negative solution 
of \eqref{fast1}--\eqref{fast2} for some suitable non-negative initial data. 
Following results by {F}ila and {F}ilo \cite{FF89, FF90},
we obtain a result of finite time extinction for some small initial data. 
\smallskip

We present a brief outline of the paper along with 
a short description of the various items. 
\smallskip

In Section 2, we state the main theorems, 
which are related to the finite time extinction
after establishing our notation. 
Let $(\lambda, \mu)=(0,1)$; for some small initial data, the unique solution of \eqref{1001} or \eqref{0101}
decays to zero in a finite time when $1/5 < m <1$.  This means that 
we can take any $q > 1$. 
On the other hand, if $(\lambda, \mu)=(1,0)$, we can obtain the same result to \eqref{1010} or \eqref{0110}
under the additional assumption $1<p < 5m $. 
\smallskip

In Section 3, we consider an auxiliary problem. 
We discuss the well-posedness of  
some globally {L}ipschitz perturbation based on {F}ilo \cite{Filo87}. 
Firstly, we set up a time discretization scheme. 
Thus, we obtain a solution for an elliptic problem applying the maximal monotone theory.
Secondly, correcting the suitable uniform estimates, we prove that 
a pair of piecewise linear functions converges to a candidate solution to the auxiliary problem. 
Using fundamental inequalities, we also obtain estimates for time derivatives since
the initial data belongs to $H^1 \cap L^\infty$.  
This is a point of emphasis, 
because the suitable regularity of the time derivative is a special property for the fast diffusion equation
(see \cite[Theorem~2]{Filo88} and Remark~3.1). 
Moreover, to obtain a regular solution, we use 
the bootstrap argument for the dynamic boundary condition. 
Thanks to surface diffusion, this argument works well. 
This is another point of emphasis because the equation is treated as a weak 
or very weak formulation of the porous media equation in general. 
The solution satisfies the equation in almost everywhere sense. 
It is a benefit of surface diffusion. 
\smallskip

In Section 4, we prove the main theorems step by step. 
Firstly, we obtain the local existence of the solution to the original problem under a general setting, 
that is, $0<m \le 1$, $p,q > 1$. We use a standard method of the cut off function. 
The solution also satisfies an energy inequality and equality of 
conservation with respect to the $L^{(1+m)/m}$-norm. 
The proof of the main theorems is based on the effective use of this inequality and equality. 
Next, under the assumption $1/5 < m <1$, 
we prove the property of the finite time extinction for 
$(\lambda, \mu)=(0,1)$, namely \eqref{1001} or \eqref{0101}. 
{\sc Figure} 1 shows the strategy of the proof of the theorem.  
To complete the proof of finite time extinction, we need to discuss the
invariance of some stable set. 
The essence of the proof is based upon {F}ila and {F}ilo \cite{FF89}. 
If $(\lambda, \mu)=(1,0)$, we need additional assumption $1<p < 5m$. 
\smallskip 

Here, let us present a detailed index of sections and subsections.
\begin{itemize}
 \item[1.] Introduction
 \item[2.] Main theorems
  \begin{itemize}
  \item[2.1.] Notation
  \item[2.2.] Main theorems
 \end{itemize}
 \item[3.] Global existence for globally Lipschitz perturbations
 \begin{itemize}
  \item[3.1.] Time discretization
  \item[3.2.] Uniform estimates
  \item[3.3.] Proof of Proposition 3.1
 \end{itemize}
 \item[4.] Proof of main theorems
  \begin{itemize}
  \item[4.1.] Finite time extinction
  \item[4.2.] Proof of invariance
 \end{itemize}
 \item[] Appendix
\end{itemize}

\section{Main theorems}

In this section, we present the main theorems. 
We first set up our problem in mathematical fundamental settings. 

\subsection{Notation}

Let $T>0$ be the finite time and $Q:=(0,T) \times \Omega$, $\Sigma:=(0,T)\times \Gamma$. 
We use the following notations: $H:=L^2(\Omega)$, 
$V:=H^1(\Omega)$, and $W:=H^2(\Omega)$, 
which are {H}ilbert spaces with standard norms $|\cdot |_X$
and inner products $(\cdot,\cdot)_X$, where $X$ is the corresponding space. 
Analogously, $H_\Gamma:=L^2(\Gamma)$, $V_\Gamma:=H^1(\Gamma)$, and 
$W_\Gamma:=H^2(\Gamma)$. 
For the pair of functions $z$ on $\Omega$ and $z_\Gamma$ on $\Gamma$, 
we use the bold character $\boldsymbol{z}:=(z,z_\Gamma)$. 
Also, we have the following definitions. 
\begin{align*}
	& \boldsymbol{H}:=H \times H_\Gamma, 
	\\
	& \boldsymbol{V}:=\bigl\{ \boldsymbol{z} \in V \times V_\Gamma \ : \  z_\Gamma = z_{|_\Gamma} \quad \mbox{a.e.\ on } \Gamma \bigr\},
	\\
	& \boldsymbol{W}:=(W \times W_\Gamma) \cap \boldsymbol{V}, 
	\\
	& \boldsymbol{L}^\infty :=L^\infty (\Omega) \times L^\infty (\Gamma).
\end{align*}
The symbol $z_{|_\Gamma}$ denotes the trace of $z$ to the boundary $\Gamma$. 
We remark that for the function $\boldsymbol{z}=(z,z_\Gamma) \in \boldsymbol{H}$, 
the first component $z$ and the second component $z_\Gamma$ are completely independent because 
of the lack of regularity. 
\smallskip

Subsequently, we set $\alpha:=1/m$ for simplicity. 
We define functions ${\rm sgn}, \gamma, g, g_\Gamma:\mathbb{R} \to \mathbb{R}$ by 
\begin{gather*}
	{\rm sgn} r :=
	\begin{cases}
	1 &  \mbox{if } r >0, \\
	0 &  \mbox{if } r=0, \\
	-1 &  \mbox{if }r <0,
	\end{cases}
	\quad 
	\gamma(r):=|r|^\alpha {\rm sgn} r =
	\begin{cases}
	r^\alpha &  \mbox{if } r >0, \\
	0 &  \mbox{if } r=0, \\
	-(-r)^\alpha &  \mbox{if }r <0,
	\end{cases}
	\\
	g(r):=|r|^{p-1}r=|r|^p {\rm sgn} r, \quad 
	g_\Gamma(r):=|r|^{q-1}r=|r|^q {\rm sgn} r.  
\end{gather*}
Moreover, we put $\beta:=\gamma^{-1}$. Then, 
$\gamma$, $\beta$, $g$, and $g_\Gamma$ are monotone functions.
Now, we can set up the problem of perturbed fast diffusion equation with a dynamic boundary condition as 
follows: 
Find $v: Q \to [0,\infty) $, $v_\Gamma: \Sigma \to [0,\infty)$ satisfying the following system
\begin{align}
	\partial _t \gamma(v) -\Delta v +a v =\lambda g\bigl( \gamma (v) \bigr)  
	& \quad \mbox{a.e.\ in } Q,
	\label{f1}\\
	v_{|_\Gamma}=v_\Gamma 
	& \quad \mbox{a.e.\ on } \Sigma, 
	\label{f2}\\
	\partial _t \gamma(v_\Gamma) + \partial _{\boldsymbol{\nu}} v 
	- \Delta_\Gamma v_\Gamma + b v_\Gamma =\mu g_\Gamma \bigl( \gamma (v_\Gamma) \bigr)  
	& \quad \mbox{a.e.\ on } \Sigma,
	\label{f3}\\
	v(0)=v_0
	& \quad \mbox{a.e.\ in } \Omega, 
	\label{f4}\\
	v_\Gamma(0)=v_{\Gamma,0}
	& \quad \mbox{a.e.\ on }\Gamma.
	\label{f5}
\end{align}
The third equation \eqref{f3} is called the dynamic boundary condition since it includes the time derivative. 
Therefore, we need two initial data for $v$ and $v_\Gamma$, or more specifically, conditions \eqref{f4} and \eqref{f5} with given 
data $v_0:\Omega \to [0, \infty)$ and $v_{\Gamma,0}:\Gamma \to  [0, \infty)$, respectively.
Compared with the previous result of {F}ilo \cite{Filo87}, the function $\boldsymbol{v}:=(v,v_\Gamma)$ will satisfy
the equations in almost everywhere sense in 
\eqref{f1} and \eqref{f3}, respectively, thanks to the presence of surface diffusion. 
In other words, we can obtain sufficient regularity.  
\smallskip

To discuss finite time extinction based on the previous results 
(see, e.g., \cite{FF89, FF90, Ish77, Lio69, Nak85, Ota81, Sat68, Tsu72}), we introduce the stable set 
${\mathcal W}$ with corresponding energy $J$ as follows: put 
\begin{equation*}
	p_*:=\lambda p+\mu q,
\end{equation*}
namely $p_*=p$ if $(\lambda, \mu)=(1,0)$ and 
$p_*=q$ if $(\lambda, \mu)=(0,1)$
\begin{gather*}
	{\mathcal W}:=\bigl\{ \boldsymbol{z} \in \boldsymbol{V} \setminus \{\boldsymbol{0} \} \ : \ 
	z \ge 0, z_\Gamma \ge 0, J(\boldsymbol{z})<d, 2 \varphi_1(\boldsymbol{z}) > (\alpha p_*+1)\varphi_2(\boldsymbol{z}) \bigr\} \cup \{ \boldsymbol{0} \}, 
	\\
	J(\boldsymbol{z}):=\varphi_1(\boldsymbol{z})-\varphi_2(\boldsymbol{z}), \\
	\varphi_1(\boldsymbol{z})
	:=
	\frac{1}{2} \int_\Omega |\nabla z|^2dx 
	+ 
	\frac{a}{2}\int_\Omega |z|^2 dx 
	+ 
	\frac{1}{2} \int_\Gamma |\nabla_\Gamma z_\Gamma|^2 d\Gamma 
	+ 
	\frac{b}{2}\int_\Gamma |z_\Gamma |^2 d\Gamma
	, \\
	\varphi_2(\boldsymbol{z}) := 
	\frac{1}{\alpha p_* +1} \left( \lambda \int_\Omega |z|^{\alpha p+1} dx + 
	\mu\int_\Gamma |z_\Gamma|^{\alpha q+1} d\Gamma \right) .
\end{gather*}
In the definition of ${\mathcal W}$, the constant $d$ is called the \emph{depth of the potential well}, and is now defined by 
\begin{equation}
	d:=\inf \bigl\{ J(\boldsymbol{z}) \ : \ 
	\boldsymbol{z} \in \boldsymbol{V} \setminus \{ \boldsymbol{0} \}, 
	2\varphi_1(\boldsymbol{z})= (\alpha p_* +1) \varphi_2(\boldsymbol{z}) \bigr\}.
	\label{d}
\end{equation}
This constant is characterized by the optimal constant of some estimate between 
$\varphi_1(\boldsymbol{z})$ and $\varphi_2(\boldsymbol{z})$, which will be discussed in Remark~4.2.

\subsection{Main theorems}

The main theorems are related to the finite time extinction for the solution 
$\boldsymbol{v}=(v,v_\Gamma)$ of \eqref{f1}--\eqref{f5}. 
Let $(a,b) \in \{(1,0), (0,1)\}$, $(\lambda, \mu)=(0,1)$. 
For the cases of \eqref{1001} or \eqref{0101}, there are no restrictions for $q > 1$. 
\begin{theorem} 
\label{ext1}
Assume that $1/5 < m <1$, $q > 1$, and $\boldsymbol{v}_0:=(v_0,v_{\Gamma,0}) 
\in {\mathcal W} \cap\boldsymbol{L}^\infty$.  
Then, there exists $T_{\rm ext} \in (0,\infty)$ depending on 
$|v_0|_{L^{(1+m)/m}(\Omega)}$ and 
$|v_{\Gamma,0}|_{L^{(1+m)/m}(\Gamma)}$ such that 
$v(t)=0$ a.e.\ in $\Omega$, $v_{\Gamma}(t)= 0$ a.e.\ on $\Gamma$ for all $t \ge T_{\rm ext}$. Moreover, 
there exists a positive constant $C(m)>0$ depending on $m$ such that 
\begin{equation}
	\bigl|v(t) \bigr|_{L^{(1+m)/m}(\Omega)}+ 
	\bigl|v_\Gamma (t) \bigr|_{L^{(1+m)/m}(\Gamma)}
	\le C(m) (T_{\rm ext}-t)^{m/(1-m)}
	 \label{upper}
\end{equation} 
for all $t \in [0,T_{\rm ext}]$. 
\end{theorem}

\begin{corollary} 
Assume that $1/5 < m <1$, $q>1$, $v_0 \in H^1_0(\Omega) \cap L^\infty(\Omega)$ with 
$v_0 \ge 0$, and 
$0< |v_0|_V^2<2d $, namely $\boldsymbol{v}_0:=(v_0, 0)$. 
Then, there exists $T_{\rm ext} \in (0,\infty)$ depending on 
$|v_0|_{L^{(1+m)/m}(\Omega)}$ 
such that 
$v(t)=0$ a.e.\ in $\Omega$, $v_{\Gamma}(t)= 0$ a.e.\ on $\Gamma$ for all $t \ge T_{\rm ext}$. 
Moreover, the same kind of estimate from above \eqref{upper} holds.
\end{corollary}

Let $(a,b) \in \{(1,0), (0,1)\}$, $(\lambda, \mu)=(1,0)$. 
For \eqref{1010} and \eqref{0110}
under the restriction of $p > 1$, we can discuss the finite time extinction.
\begin{theorem} 
\label{ext2}
Assume that $1/5 < m <1$, $1 < p < 5m$, and 
$\boldsymbol{v}_0 \in {\mathcal W} \cap \boldsymbol{L}^\infty$.  
Then, there exists $T_{\rm ext} \in (0,\infty)$ depending on 
$|v_0|_{L^{(1+m)/m}(\Omega)}$ and 
$|v_{\Gamma,0}|_{L^{(1+m)/m}(\Gamma)}$ such that 
$v(t)=0$ a.e.\ in $\Omega$, $v_{\Gamma}(t)= 0$ a.e.\ on $\Gamma$ for all $t \ge T_{\rm ext}$. 
Moreover, the same kind of estimate from above \eqref{upper} holds.
\end{theorem}

As a remark, the well-posedness of the problem is discussed in Proposition~4.1. 

\section{Global existence for globally Lipschitz perturbations}

In this section, we discuss the auxiliary problem for \eqref{f1}--\eqref{f5}. 
Let us replace the perturbations $g$, $g_\Gamma$ by 
globally {L}ipschitz continuous monotone functions $f, f_\Gamma$ with {L}ipschitz constants $L_f, L_{f_\Gamma}>0$.  
Furthermore, $f$ and $f_\Gamma$ satisfy $f(0)=f_\Gamma(0)=0$:  
Find $v: Q \to [0,\infty) $, $v_\Gamma: \Sigma \to [0,\infty)$ satisfying the system
\begin{align}
	\partial _t \gamma(v) -\Delta v +a v =f \bigl( \gamma (v) \bigr)  & \quad \mbox{a.e.\ in } Q,
	\label{af1}\\
	v_{|_\Gamma}=v_\Gamma & \quad \mbox{a.e.\ on } \Sigma, 
	\label{af2}\\
	\partial _t \gamma(v_\Gamma) + \partial _{\boldsymbol{\nu}} v 
	- \Delta_\Gamma v_\Gamma +bv_\Gamma =f_\Gamma \bigl( \gamma (v_\Gamma) \bigr)  
	& \quad \mbox{a.e.\ on } \Sigma,
	\label{af3}\\
	v(0)=v_0
	& \quad \mbox{a.e.\ in } \Omega, 
	\label{af4}\\
	v_\Gamma(0)=v_{\Gamma,0}
	& \quad \mbox{a.e.\ on }\Gamma.
	\label{af5}
\end{align}
In this section, we set $(a,b) \in \{(1,0), (0,1)\}$, and 
the function $\gamma$ is the same as it was in the previous section. 
\smallskip

We obtain the global existence and uniqueness result for globally Lipschitz perturbations as follows.

\begin{proposition}
\label{global}
Let $0<m \le 1$, $0<T<\infty$. 
Let us assume that 
$\boldsymbol{v}_0:=(v_0, v_{\Gamma,0}) \in \boldsymbol{V} \cap \boldsymbol{L}^\infty$ with 
$v_0 \ge 0$ and $v_{\Gamma,0} \ge 0$. Then, there exists a unique pair of non-negative functions 
$\boldsymbol{v}:=(v,v_\Gamma)$ such that   
\begin{align*}
	& v \in C \bigl( [0,T]; H \bigr) \cap L^\infty \bigl( 0,T;V \cap L^\infty (\Omega)\bigr) \cap L^2(0,T;W), \\
	& v^{\tiny (\alpha+1)/2} =v^{(1+m)/2m}\in H^1(0,T;H), \\
	& \gamma(v) \in H^1(0,T;H) \cap L^\infty ( 0,T;V ),\\
	& v_\Gamma \in C \bigl( [0,T]; H_\Gamma \bigr) 
	\cap L^\infty \bigl( 0,T;V_\Gamma \cap L^\infty(\Gamma) \bigr) 
	\cap L^2(0,T;W_\Gamma), \\
	& v_\Gamma^{(\alpha+1)/2} =v_\Gamma ^{(1+m)/2m}\in H^1(0,T;H_\Gamma), \\
	& \gamma(v_\Gamma) \in H^1(0,T;H_\Gamma) \cap  L^\infty ( 0,T;V_\Gamma ) 
\end{align*}
and \eqref{af1}--\eqref{af5} hold. 
Moreover, they satisfy the energy inequality
\begin{align}
	& \frac{4\alpha}{(\alpha+1)^2} \int_0^t 
	\left( \int_\Omega \bigl| \partial_t \bigl(v^{(\alpha+1)/2}(s) \bigr)\bigr|^2 dx 
	+  \int_\Gamma \bigl| \partial_t \bigl( v_\Gamma^{(\alpha+1)/2}(s) \bigr) \bigr|^2 d\Gamma 
	\right) ds 
	\nonumber \\
	 & \quad {}+ \varphi_1\bigl(\boldsymbol{v}(t) \bigr) 
	 - \int _\Omega \widehat{f_\gamma } \bigl( v(t) \bigr) dx
	 - \int _\Gamma \widehat{f}_{\Gamma,\gamma} \bigl( v_\Gamma(t)  \bigr) d\Gamma
	 \nonumber \\
	& \le \varphi_1 (\boldsymbol{v}_0 ) 
	- \int _\Omega \widehat{f}_\gamma (v_0) dx
	- \int _\Gamma \widehat{f}_{\Gamma,\gamma} (v_{\Gamma,0}) d\Gamma
	\label{aenergy}
\end{align}
and the $L^\infty$-boundednesses
\begin{align}
	\bigl| v(t) \bigr|_{L^\infty(\Omega)} 
	\le 
	e^{L t/\alpha}
	\bigl( | v_0 |_{L^\infty(\Omega)} + | v_{\Gamma,0} |_{L^\infty(\Gamma)}\bigr), 
	\label{alinfty1}\\
	\bigl| v_\Gamma(t) \bigr|_{L^\infty(\Gamma)} \le 
	e^{L t/\alpha}
	\bigl( | v_0 |_{L^\infty(\Omega)} + | v_{\Gamma,0} |_{L^\infty(\Gamma)}\bigr)
	\label{alinfty2}
\end{align}
for all $t \in [0,T]$, where $L:=\max\{L_f, L_{f_\Gamma}\}$. Furthermore, there exists a positive constant $M_0$ such that 
\begin{equation}
	\bigl| v(t)-v(s) \bigr|_H + \bigl| v_\Gamma (t)-v_\Gamma (s) \bigr|_H \le M_0 |t-s|^{1/(\alpha+1)}
	\label{aholder}
\end{equation}
for all $s, t \in [0,T]$. 
\end{proposition}

We present the proof of the proposition in Subsection~3.3. 
In estimate \eqref{aenergy}, the function $\widehat{f}_\gamma$ is defined as the primitive of $f\circ \gamma $, namely 
\begin{equation*}
	\widehat{f}_\gamma (r) := \int_0^r (f \circ \gamma) (s) ds = \int_0^r f \bigl( \gamma (s) \bigr) ds \quad \mbox{for all } r \in \mathbb{R}.
\end{equation*}
The primitive $\widehat{f}_{\Gamma,\gamma}$ of $f_\Gamma \circ \gamma $ is also defined analogously.  
\smallskip

To discuss the existence of solutions to the above problem, we employ the argument of {F}ilo \cite{Filo87}. 
Indeed, the well-posedness for the nonlinear diffusion equation with the 
dynamic boundary condition without the perturbation, can be solved using idea from 
\cite{Fuk16, Fuk16b, FM18, SSZ16}. 
To this problem, see also the abstract approach of evolution equations governed by the difference between two subdifferentials 
\cite{Aka09, AO05, Ish77, KW76, Ota77}. 
Also we consider \cite{Gal12} and that author's series of papers on various problems with dynamic boundary conditions. 

\subsection{Time discretization}
The essential idea of {F}ilo \cite{Filo87} was to apply time discretization and suitable 
fundamental inequalities. Let $n \in \mathbb{N}$, and 
set $h:=T/n$: 
for each $i=1,2,\ldots, n$, find $v_{i}$ and $v_{\Gamma,i}$ satisfying 
\begin{align}
	\frac{\gamma(v_{i})-\gamma(v_{i-1})}{h} - \Delta v_{i} +a v_{i} =f_{i-1} 
	& \quad \mbox{a.e.\ in } \Omega,
	\label{ap1}\\
	(v_{i}){|_\Gamma}=v_{\Gamma ,i} 
	& \quad \mbox{a.e.\ on } \Gamma, 
	\label{ap2}\\
	\frac{\gamma(v_{\Gamma,i})-\gamma(v_{\Gamma,i-1})}{h} + \partial _{\boldsymbol{\nu}} v_{i} 
	- \Delta_\Gamma v_{\Gamma, i} + b v_{\Gamma, i}= f _{\Gamma,i-1}
	& \quad \mbox{a.e.\ on } \Gamma,
	\label{ap3}
\end{align}
where $ f_{i-1}:=f( \gamma(v_{i-1}) )$ and $f_{\Gamma, i-1}:=f_\Gamma ( \gamma(v_{\Gamma,i-1}) )$. 
Then, we see that 
there exists a unique pair $(v_{i}, v_{\Gamma,i}) \in \boldsymbol{W}$ of 
non-negative functions 
such that \eqref{ap1}--\eqref{ap3} holds for all $i=1,2,\ldots, n$. 
Indeed, we define an operator $\boldsymbol{A}:D(\boldsymbol{A}) \to \boldsymbol{H}$ by 
$\boldsymbol{A}\boldsymbol{z}:=(\gamma(z),\gamma(z_\Gamma))$ with 
$D(\boldsymbol{A})=L^{2\alpha}(\Omega) \times L^{2\alpha}(\Gamma)$. 
Then, $\boldsymbol{A}^{-1}$ is monotone and hemi-continuous. Therefore, 
$\boldsymbol{A}$ is maximal monotone same as $\boldsymbol{A}^{-1}$ (see, e.g., \cite[p.36, Theorem~2.4, p.29, Proposition~2.1]{Bar10}). 
Moreover, we 
define a proper, lower semi-continuous, and convex functional 
$\varphi_{\boldsymbol{H}} : \boldsymbol{H} \to [0,\infty]$ by 
\begin{equation*}
	\varphi_{\boldsymbol{H}}(\boldsymbol{z}):=
	\begin{cases}
	\varphi_1(\boldsymbol{z}) & \mbox{if } \boldsymbol{z} \in \boldsymbol{V}, \\
	\infty & \mbox{if } \boldsymbol{z} \in \boldsymbol{H} \setminus \boldsymbol{V}.
	\end{cases}
\end{equation*}
Then, we see that the subdifferential $\partial \varphi_{\boldsymbol{H}}$ is a maximal monotone operator, 
characterized by 
$\partial \varphi_{\boldsymbol{H}} (\boldsymbol{z})=(-\Delta z + az, \partial_{\boldsymbol{\nu}}z-\Delta_\Gamma z_\Gamma+b z_\Gamma)$ with 
domain $D(\partial \varphi_{\boldsymbol{H}})=\boldsymbol{W}$ (see, e.g., \cite{Bar10, Bre73, CF15}). 
Thanks to the standard maximal monotone theory (see, e.g., \cite[p.44, Theorem~2.7]{Bar10}), 
$A + \partial \varphi_{\boldsymbol{H}}$ is also maximal monotone. 
Moreover, there exists a positive constant $C_{\rm C}>0$ such that
\begin{align}
	\bigl( \boldsymbol{A} \boldsymbol{z} + \partial \varphi_{\boldsymbol{H}}(\boldsymbol{z}), \boldsymbol{z} \bigr)_{\boldsymbol{H}}
	& \ge \int_\Omega |\nabla z|^2 dx 
	+ a  \int_\Omega |z|^2 dx 
	+ \int_\Gamma |\nabla_\Gamma z_\Gamma|^2 d\Gamma 
	+ b \int_\Gamma |z_\Gamma|^2 d\Gamma 
	\nonumber \\
	& \Bigl( \ =2 \varphi_1(\boldsymbol{z}) \ \Bigr) 
	\nonumber \\
	& \ge C_{\rm C} |\boldsymbol{z}|_{\boldsymbol{V}}^2
	\label{coercive}
\end{align}
for all $\boldsymbol{z} \in \boldsymbol{W}$.
Indeed, 
if $(a,b)=(0,1)$, the {P}oincar\'e inequality ensures that  
there exists a positive constant $C_{\rm P}>0$ such that 
\begin{equation*}
	|z|_{V}^2 \le C_{\rm P} \left( \int_\Omega |\nabla z|^2 dx 
	+ b \int_\Gamma |z_{|_\Gamma}|^2 d\Gamma  \right)
	\quad \mbox{for all } z \in V.
	\label{poin}
\end{equation*}
If $(a,b)=(1,0)$, the trace theory between $V$ and $H_\Gamma$ ensures that 
there exists a positive constant $C_{\rm T}>0$ such that 
\begin{equation*}
	|z_{|_\Gamma} |_{H_\Gamma}^2 \le C_{\rm T} \left( \int_\Omega |\nabla z|^2 dx 
	+ a \int_\Omega |z|^2 dx \right) \quad \mbox{for all } z \in V.
	\label{trace}
\end{equation*}
Therefore, $\boldsymbol{A} + \partial \varphi_{\boldsymbol{H}}$ is coercive. 
Thus, 
the range $R(\boldsymbol{A}+\partial \varphi_{\boldsymbol{H}})$ of $\boldsymbol{A}+\partial \varphi_{\boldsymbol{H}}$ is 
the whole space $\boldsymbol{H}$ (see, e.g., \cite[p.36, Corollary 2.2]{Bar10}). 
Next, multiplying \eqref{ap1} by $\min\{0, v_{i} \} \in V$ and  
\eqref{ap3} by $\min\{0, v_{\Gamma, i} \} \in V _\Gamma$, respectively, 
we obtain the non-negativity of the functions $v_{i}$ and $v_{\Gamma,i}$. 

\subsection{Uniform estimates} 
According to \cite[Lemma 1.14]{Filo87}, we obtain the $L^\infty$-boundedness as follows:

\begin{lemma} 
\label{linf}
The functions $v_{i}$ and $v_{\Gamma,i}$ satisfy 
\begin{align*}
	|v_{i}|_{L^\infty(\Omega)}
	\le (1+L h)^{i/\alpha} \bigl( |v_0|_{L^\infty(\Omega)} + |v_{\Gamma,0}|_{L^\infty(\Gamma)} \bigr),
	\\
	|v_{\Gamma,i}|_{L^\infty(\Gamma)} 
	\le (1+L h)^{i/\alpha} \bigl( |v_0|_{L^\infty(\Omega)} + |v_{\Gamma,0}|_{L^\infty(\Gamma)} \bigr). 
\end{align*}
for all $i=1,2,\ldots, n$. 
\end{lemma}

\begin{proof}
Let $\kappa >1$. 
Multiplying \eqref{ap1} by $v_{i}^\kappa $ and \eqref{ap3} by $v_{\Gamma, i}^\kappa$, 
using \eqref{ap2}, and 
summing the results, 
we deduce that
\begin{align*}
	& \int_\Omega v_{i}^{\alpha+\kappa} dx + \int_\Gamma v_{\Gamma, i}^{\alpha+\kappa} d\Gamma 
	+ \biggl( h \kappa \int_\Omega v_{i}^{\kappa-1}|\nabla v_{i}|^2  dx 
	+ h \kappa \int_\Gamma v_{\Gamma, i}^{\kappa-1}|\nabla _\Gamma v_{\Gamma, i}|^2 d\Gamma 
	\nonumber \\
	& \quad {}
	+ h a \int_\Omega v_{i}^{\kappa+1}  dx 
	+ h b \int_\Gamma v_{\Gamma, i}^{\kappa+1} d\Gamma \biggr)
	\nonumber \\
	& \le  \int_\Omega v_{i-1}^\alpha v_{i}^{\kappa}  dx 
	+  \int_\Gamma v_{\Gamma, i-1}^\alpha v_{\Gamma,i}^{\kappa} d\Gamma 
	+ h \int_\Omega f_{i-1} v_{i}^{\kappa}  dx 
	+ h \int_\Gamma f_{\Gamma,i-1} v_{\Gamma,i}^{\kappa} d\Gamma 
	\nonumber \\
	& \le (1+L_f h) \int_\Omega v_{i-1}^\alpha v_{i}^{\kappa}  dx 
	+  (1+L_{f_\Gamma}h) \int_\Gamma v_{\Gamma, i-1}^\alpha v_{\Gamma,i}^{\kappa} d\Gamma 
	\nonumber \\
	& \le  
	\frac{\kappa}{\alpha+\kappa} \int_\Omega  v_{i}^{\kappa \cdot \frac{\alpha+\kappa}{\kappa}}  dx 
	+\frac{\alpha}{\alpha+\kappa}
	(1+L h)^{(\alpha+\kappa)/\alpha} \int_\Omega v_{i-1}^{\alpha \cdot \frac{\alpha+\kappa}{\alpha}}  dx 
	\nonumber \\
	& \quad {}
	+ \frac{\kappa}{\alpha+\kappa} \int_\Gamma  v_{\Gamma, i}^{\kappa \cdot \frac{\alpha+\kappa}{\kappa}}  d\Gamma 
	+\frac{\alpha}{\alpha+\kappa}
	(1+L h)^{(\alpha+\kappa)/\alpha} \int_\Gamma v_{\Gamma,i-1}^{\alpha \cdot \frac{\alpha+\kappa}{\alpha}} d\Gamma 
\end{align*}
for all $i=1,2,\ldots, n$, 
where we used the Young inequality. 
Now, the second terms of the left hand side are non-negative. 
Therefore, we use the above estimate recurrently: 
\begin{equation*}
	\int_\Omega v_{i}^{\alpha+\kappa} dx + \int_\Gamma v_{\Gamma, i}^{\alpha+\kappa} d\Gamma 
	\le  
	(1+L h)^{i (\alpha+\kappa)/\alpha} 
	\left( \int_\Omega v_{0}^{\alpha+\kappa}  dx 
	+ \int_\Gamma v_{\Gamma,0}^{\alpha+\kappa} d\Gamma \right).
\end{equation*}
This implies that 
\begin{align*}
	|v_{i}|_{L^{\alpha+\kappa}(\Omega)} 
	\le (1+L h)^{i/\alpha} 
	\left(|v_{0}|_{L^{\alpha+\kappa}(\Omega)} + |v_{\Gamma,0}|_{L^{\alpha+\kappa}(\Gamma)}  \right),
	\\
	 |v_{\Gamma, i}|_{L^{\alpha+\kappa}(\Gamma)} 
	\le (1+L h)^{i/\alpha} 
	\left(|v_{0}|_{L^{\alpha+\kappa}(\Omega)} + |v_{\Gamma,0}|_{L^{\alpha+\kappa}(\Gamma)}  \right)
\end{align*}
for all $i=1,2,\ldots, n$. 
Thus, letting $\kappa \to \infty$ we find the conclusion. 
\end{proof}

Using Lemma \ref{linf}, we obtain the following estimates:
\begin{lemma}
\label{ests}
There exist positive constants $M_1$, $M_2$, and $M_3>0$ independent of $n \in \mathbb{N}$ such that
\begin{gather}
	\sum_{i=1}^{n} \left| \frac{v_{i}^{(\alpha+1)/2}-v_{i-1}^{(\alpha+1)/2}}{h}
	\right|_{H}^2 h 
	+\sum_{i=1}^{n} \left| \frac{v_{\Gamma,i}^{(\alpha+1)/2}-v_{\Gamma, i-1}^{(\alpha+1)/2}}{h}
	\right|_{H_\Gamma }^2 h 
	\le M_1, 
	\label{apone}	
	\\
	|v_{i}|_V+ |v_{\Gamma,i}|_{V_\Gamma} \le M_2 \quad \mbox{for all } i=1,2,\ldots,n,
	\label{aptwo}
	\\
	\sum_{i=1}^{n} \left| \frac{\gamma(v_{i})-\gamma(v_{i-1})}{h}
	\right|_{H}^2 h 
	+\sum_{i=}^{n} \left| \frac{\gamma(v_{\Gamma,i})-\gamma(v_{\Gamma, i-1})}{h}
	\right|_{H_\Gamma }^2 h 
	\le M_3
	\label{apthree}
\end{gather} 
for all $n \in \mathbb{N}$. 
\end{lemma}

\begin{proof}
Multiplying \eqref{ap1} by $v_{i}-v_{i-1}$ and \eqref{ap3} by $v_{\Gamma, i}-v_{\Gamma,i-1}$, 
using \eqref{ap2}, and 
summing these results,  
we deduce that
\begin{align*}
	& \int_\Omega  \frac{\gamma(v_{i})-\gamma(v_{i-1})}{h} (v_{i}-v_{i-1})dx 
	+ \int_\Gamma \frac{\gamma(v_{\Gamma,i})-\gamma(v_{\Gamma, i-1})}{h}(v_{\Gamma, i}-v_{\Gamma,i-1})d\Gamma
	\nonumber \\
	 & \quad {}+ \varphi_1 ( \boldsymbol{v}_{i}) -\varphi_1 ( \boldsymbol{v}_{i-1}) 
	 \nonumber \\
	& \le 
	\int _\Omega \widehat{f}_\gamma(v_{i})dx 
	-\int _\Omega \widehat{f}_\gamma(v_{i-1})dx
	+
	\int _\Gamma\widehat{f}_{\Gamma, \gamma} (v_{\Gamma, i}) d\Gamma 
	-
	\int _\Gamma\widehat{f}_{\Gamma, \gamma} (v_{\Gamma, i-1}) d\Gamma 
\end{align*}
for all $i=1,2,\ldots, n$. 
Now, recall the fundamental inequality 
\begin{equation*}
	\frac{4\alpha}{(\alpha+1)^2} 
	\bigl( r^{(\alpha+1)/2}-s^{(\alpha+1)/2} \bigr)^2 \le (r^\alpha-s^\alpha)(r-s)
\end{equation*}
for all $r,s \ge 0$ (see e.g., \cite[Proposition 2]{Filo88} and Appendix). 
We obtain
\begin{align}
	& 
	\frac{4\alpha}{(\alpha+1)^2}
	\left(
	\left| \frac{v_{i}^{(\alpha+1)/2}-v_{i-1}^{(\alpha+1)/2}}{h}
	\right|_{H}^2 h 
	+ \left| \frac{v_{\Gamma,i}^{(\alpha+1)/2}-v_{\Gamma, i-1}^{(\alpha+1)/2}}{h}
	\right|_{H_\Gamma }^2 h 
	\right)
	\nonumber \\
	 & \quad {}+ \varphi_1 ( \boldsymbol{v}_{i}) 
	 -
	  \int _\Omega \widehat{f} _ \gamma (v_{i})  dx
	  -
	   \int _\Gamma \widehat{f}_{\Gamma,\gamma} (v_{\Gamma,i})d\Gamma
	 \nonumber \\
	& \le 
	\varphi_1 ( \boldsymbol{v}_{i-1}) 
	-
	\int _\Omega \widehat{f}_\gamma (v_{i-1}) dx
	-
	\int _\Gamma \widehat{f}_{\Gamma,\gamma}  (v_{\Gamma,i-1})  d\Gamma
	\label{time2}
\end{align}
for all $i=1,2,\ldots, n$. 
Summing \eqref{time2} from $i=1$ to $i=j \le n$, we obtain 
\begin{align}
	& 
	\frac{4\alpha}{(\alpha+1)^2} 
	\left(
	\sum_{i=1}^{j} \left| \frac{v_{i}^{(\alpha+1)/2}-v_{i-1}^{(\alpha+1)/2}}{h}
	\right|_{H}^2 h 
	+ 
	\sum_{i=1}^{j}
	\left| \frac{v_{\Gamma,i}^{(\alpha+1)/2}-v_{\Gamma, i-1}^{(\alpha+1)/2}}{h}
	\right|_{H_\Gamma }^2 h 
	\right)
	\nonumber \\
	 & \quad {}+ \varphi_1 ( \boldsymbol{v}_{j}) 
	 -
	  \int _\Omega \widehat{f}_ \gamma (v_{j}) dx
	  -
	   \int _\Gamma \widehat{f}_{\Gamma, \gamma} (v_{\Gamma,j}) d\Gamma
	 \nonumber \\
	& \le 
	\varphi_1 ( \boldsymbol{v}_{0}) 
	-
	\int _\Omega \widehat{f}_\gamma (v_{0} ) dx
	-
	\int _\Gamma \widehat{f}_{\Gamma, \gamma} (v_{\Gamma,0}) d\Gamma.
	\label{enelast}
\end{align}
Here, using the {L}ipschitz continuities of $f$ and $f_\Gamma$, we have
\begin{equation*}
	\bigl| \widehat{f}_\gamma (r) \bigr| \le \int_0^r \bigl|f \bigl( \gamma(s) \bigr ) \bigr| ds \le L_f\int_0^r s^\alpha \, ds 
	= \frac{L}{\alpha+1}r ^{\alpha+1}, 
	\quad \bigl|\widehat{f}_\Gamma (r) \bigr| \le \frac{L}{\alpha+1}r ^{\alpha+1}
\end{equation*}
for all $r \ge 0$.
Therefore, applying Lemma \ref{linf}, we obtain that there exists a positive constant $\tilde{M}_1>0$ depending on 
$\alpha, T, |v_0|_{L^\infty(\Omega)}, |v_{\Gamma,0}|_{L^\infty(\Gamma)}$, and $L$, independent of $n$ such that
\begin{align}
	& \int _\Omega \widehat{f}_\gamma ( v_{j}) dx
	+
	\int _\Gamma \widehat{f}_{\Gamma, \gamma} (v_{\Gamma,j}) d\Gamma
	\nonumber \\
	& \le 
	 \frac{L}{\alpha+1} (1+Lh)^{(\alpha+1)j/\alpha}
	 \left( 
	 |v_0|_{L^\infty(\Omega)}
	 + |v_{\Gamma,0}|_{L^\infty(\Gamma)} \right)^{\alpha+1} \bigl( |\Omega| + |\Gamma| \bigr)
	\nonumber\\
	& \le \frac{L}{\alpha+1} e^{LT(\alpha+1)/\alpha}
	 \left( 
	 |v_0|_{L^\infty(\Omega)}
	 + |v_{\Gamma,0}|_{L^\infty(\Gamma)} \right)^{\alpha+1} \bigl( |\Omega| + |\Gamma| \bigr)=:\tilde{M}_1.
	\label{fff}
\end{align}
Thus, we deduce that 
\begin{align*}
	& 
	\sum_{i=1}^{n} \left| \frac{v_{i}^{(\alpha+1)/2}-v_{i-1}^{(\alpha+1)/2}}{h}
	\right|_{H}^2 h 
	+ 
	\sum_{i=1}^{n}
	\left| \frac{v_{\Gamma,i}^{(\alpha+1)/2}-v_{\Gamma, i-1}^{(\alpha+1)/2}}{h}
	\right|_{H_\Gamma }^2 h 
	 \nonumber \\
	& \le 
	\frac{(\alpha+1)^2}{4\alpha}
	\left( 
	\varphi_1 ( \boldsymbol{v}_{0}) 
	+ \tilde{M}_1
	  \right)=:M_1
\end{align*}
for all $n \in \mathbb{N}$;
that is,  
we obtain equation \eqref{apone}.  
Next, using equations \eqref{coercive}, \eqref{enelast}, and \eqref{fff},
we obtain \eqref{aptwo}. 
To obtain \eqref{apthree}, we apply the fundamental inequality 
\begin{equation}
	|r^\alpha-s^\alpha|
	\le 	
	\frac{2\alpha}{\alpha+1} 
	\max\{ r,s \}^{(\alpha-1)/2}
	\bigl| r^{(\alpha+1)/2}-s^{(\alpha+1)/2} \bigr|
	\label{last}
\end{equation}
for all $r,s \ge 0$ since $\alpha \ge 1$ (see, for example, \cite[Lemma 1.20]{Filo87}, 
a similar method in \cite[p.477]{Lio69}, and the Appendix). From \eqref{last} and Lemma~\ref{linf} we have 
\begin{align*}
	& \int_\Omega  \left| \frac{\gamma(v_{i})-\gamma(v_{i-1})}{h} \right| ^2 dx 
	\nonumber \\
	& \le \max \bigl\{ |v_i|_{L^\infty(\Omega)}, |v_{i-1}|_{L^\infty(\Omega)} \bigr\}^{\alpha-1}
	\left( \frac{2\alpha}{\alpha+1}\right)^2
	\int _\Omega 
	\left| \frac{v_{i}^{(\alpha+1)/2}-v_{i-1}^{(\alpha+1)/2}}{h}
	\right|^2 dx 
	\nonumber \\
	& \le 
	 e^{(\alpha-1)LT/\alpha} 
	\left(|v_{0}|_{L^\infty(\Omega)} + |v_{\Gamma,0}|_{L^\infty (\Gamma)}  \right)^{\alpha-1}
	\left( \frac{2\alpha}{\alpha+1}\right)^2
	\left| \frac{v_{i}^{(\alpha+1)/2}-v_{i-1}^{(\alpha+1)/2}}{h}
	\right|_H^2
\end{align*}
for all $i=1,2,\ldots, n$. 
Thus, we obtain the same estimate for $\gamma(v_{\Gamma,i})$, multiplying $h$, and 
summing up them. 
Moreover, summing these results from $i=1$ to $n$, then we apply \eqref{apone} to deduce \eqref{apthree}. 
\end{proof}

\begin{remark}
The last estimate \eqref{apthree} can be obtained only in the cases of the fast diffusion equation for $0<m<1$ or the 
heat equation for $m=1$. Indeed, the inequality \eqref{last}
is true for $\alpha \ge 1$. 
However, it is the advantage from the $L^\infty$-bounded initial data. 
For this reason, the time derivative
is treated in the dual space of $V$ or $V_\Gamma$ in the case of the porous media equation for $m >1$. 
In the case of the fast diffusion equation with {D}irichlet boundary condition and $H^1_0 \cap L^{\alpha+1}$-bounded initial data, 
Akagi--Kajikiya obtained crucial results in \cite{AK13}. 
\end{remark}

\subsection{Proof of Proposition 3.1}
In the standard manner, we define the following piecewise linear functions and step functions:
\begin{align*}
	\hat v_h(t)& :=v_{i-1} + \frac{v_{i}-v_{i-1}}{h}(t-ih) \quad & \mbox{for } t \in \bigl[ (i-1)h, ih \bigr], \\
	\bar v_h(t)& :=v_i \quad & \mbox{for } t \in \bigl( (i-1)h, ih \bigr], \\
	\hat v_h^*(t)& :=\gamma(v_{i-1}) + \frac{\gamma(v_{i})-\gamma(v_{i-1})}{h}(t-ih) \quad & \mbox{for } t \in \bigl[ (i-1)h, ih \bigr], \\
	\bar v_h^* (t)& :=\gamma(v_i) \quad & \mbox{for } t \in \bigl( (i-1)h, ih \bigr], \\
	\underline{v}_h^* (t)& :=\gamma(v_{i-1}) \quad & \mbox{for } t \in \bigl[ (i-1)h, ih \bigr), \\
	\hat v_h^{**}(t)& :=v_{i-1}^{(\alpha+1)/2} + \frac{v_{i}^{(\alpha+1)/2}-v_{i-1}^{(\alpha+1)/2}}{h}(t-ih) \quad & \mbox{for } t \in \bigl[ (i-1)h, ih \bigr], \\
	\bar v_h^{**}(t)& :=v_i^{(\alpha+1)/2} \quad & \mbox{for } t \in \bigl( (i-1)h, ih \bigr], 
\end{align*}
for $i=1,2,\ldots,n$, 
and analogously for $\hat v_{\Gamma,h}, \bar v_{\Gamma,h}, \hat v_{\Gamma,h}^*, \bar v_{\Gamma,h}^*, \underline{v}_{\Gamma,h}^*,\hat v_{\Gamma,h}^{**}, \bar v_{\Gamma,h}^{**}$. 
According \eqref{ap1}--\eqref{ap3}, these functions satisfy the following equations:
\begin{align}
	\partial_t \hat{v}_h^* - \Delta \bar{v}_h +a \bar{v}_h =f(\underline{v}^*_{h})
	& \quad \mbox{a.e.\ in } \Omega,
	\label{ap1b}\\
	(v_{h}){|_\Gamma}=v_{\Gamma ,h} 
	& \quad \mbox{a.e.\ on } \Gamma, 
	\label{ap2b}\\
	\partial_t \hat{v}_{\Gamma, h}^* + \partial _{\boldsymbol{\nu}} \bar v_{h} 
	- \Delta_\Gamma \bar v_{\Gamma, h} + b \bar v_{\Gamma, h}= f _{\Gamma}(\underline{v}_{\Gamma,h}^*)
	& \quad \mbox{a.e.\ on } \Gamma, 
	\label{ap3b}\\
	\hat{v}_h(0)=v_0 
	& \quad \mbox{a.e.\ in }\Omega, 
	\label{ap4b}\\
	\hat{v}_{\Gamma, h}(0)=v_{\Gamma,0} 
	& \quad \mbox{a.e.\ on }\Gamma. 
	\label{ap5b}
\end{align}
Here, we have the following useful properties:
\begin{align}
	& |\hat v_h^* |_{L^2(0,T;X)}^2 \le \frac{h}{2} \bigl| \gamma(v_0) \bigr|_X^2+|\bar v_h^* |_{L^2(0,T;X)}^2, 
	\label{tool1}
	\\
	& | \hat v_h^* |_{L^\infty(0,T;X)} = \max\bigl\{ \bigl| \gamma(v_0) \bigr|_X, |\bar v_h^* |_{L^\infty(0,T;X)} \bigr\}, 
	\label{tool2}
	\\
	& |\hat v_h^* - \bar v_h^* |_{L^2(0,T;X)}^2 =|\hat v_h^* - \underline{v}^*_{h} |_{L^2(0,T;X)}^2 = \frac{h^2}{3} |\partial _t \hat v_h^* |_{L^2(0,T;X)}^2,
	\label{tool3}
\end{align}
for some suitable function space $X$. 
\smallskip

Now, we prove Proposition 3.1.

\begin{proof} Thanks to \eqref{aptwo}, \eqref{apthree}, \eqref{tool1}, and \eqref{tool2},  
we obtain uniform estimates 
for $\bar{\boldsymbol{v}}_h:=(\bar{v}_h,\bar{v}_{\Gamma,h})$, 
$\hat{\boldsymbol{v}}^*_h:=(\hat{v}^*_h,\hat{v}_{\Gamma,h}^*)$, 
and 
$\bar{\boldsymbol{v}}^*_h:=(\bar{v}^*_h,\bar{v}_{\Gamma,h}^*)$ for example,
\begin{gather*}
	|\bar{v}_h|_{L^\infty(0,T;V)} \le M_2, \\
	|\partial_t \hat{v}^*_h|_{L^2(0,T;H)}^2 \le M_3, \\
		\bigl| \hat{v}_h(t) \bigr|_{L^\infty(\Omega)} 
	\le e^{L T/\alpha}\left( |v_0|_{L^\infty(\Omega)}+ |v_{\Gamma,0}|_{L^\infty(\Gamma)} \right)
	\quad {\rm for~all~} t \in [0,T],
	\\
		\bigl| \hat{v}_h^*(t) \bigr|_{L^\infty(\Omega)} 
\le e^{L T}\left( |v_0|_{L^\infty(\Omega)}+ |v_{\Gamma,0}|_{L^\infty(\Gamma)} \right)^\alpha
\quad {\rm for~all~} t \in [0,T],\\
	|\bar{v}_h^*|_{L^\infty(0,T;V)} = \max_{i=1,2,\ldots,n} \bigl| \gamma(v_i) \bigr| _{V} \le M_4, \\  
	|\hat{v}^*_h|_{L^\infty (0,T;V)} = \max_{i=0,1,\ldots,n} \bigl| \gamma(v_i) \bigr| _{V} \le M_4, \\ 
	|\hat{v}^*_h|_{L^2(0,T;H)}^2 \le \frac{h}{2} \bigl|\gamma(v_0) \bigr|_H^2+|\bar{v}_h^*|_{L^2(0,T;H)}^2 \le M_5,
\end{gather*}
where $M_4$ and $M_5$ are positive constants independent of $n \in \mathbb{N}$. 
Here, we used Lemmas~\ref{linf} and \ref{ests}, as well as $\nabla \gamma (v_i)=\alpha v_i^{\alpha-1} \nabla v_i$,
Then, we see that there exist  
a subsequence $\{h_k \}_{k \in \mathbb{N}}$ and 
functions 
$ \boldsymbol{v} =(v,v_\Gamma ) \in L^\infty(0,T;\boldsymbol{V} \cap \boldsymbol{L}^\infty)$, 
$\boldsymbol{v}^* =(v^*,v_\Gamma^*) \in H^1(0,T;\boldsymbol{H}) \cap 
L^\infty(0,T;\boldsymbol{V} \cap \boldsymbol{L}^\infty)$ such that 
\begin{align*}
	\bar{\boldsymbol{v}}_{h_k} \to \boldsymbol{v} & \quad \mbox{weakly star in } L^\infty(0,T;\boldsymbol{V}),  \\
	\hat{v}_{h_k} \to v & \quad \mbox{a.e.\ in } Q, \\
	\hat{v}_{\Gamma, h_k} \to v_\Gamma & \quad \mbox{a.e.\ on }\Sigma, \\
	\hat{\boldsymbol{v}}_{h_k}^* \to \boldsymbol{v}^* & \quad \mbox{weakly in } H^1(0,T;\boldsymbol{H}),  \\
	& \quad \mbox{weakly star in } L^\infty(0,T;\boldsymbol{V}),  \\
	\bar{\boldsymbol{v}}_{h_k}^* \to \boldsymbol{v}^* & \quad \mbox{weakly star in } L^\infty(0,T;\boldsymbol{V}),\\
	\hat{v}_{h_k}^* \to v^* & \quad \mbox{a.e.\ in } Q, \\
	\hat{v}_{\Gamma, h_k}^* \to v^*_\Gamma & \quad \mbox{a.e.\ on }\Sigma
\end{align*}
as $k \to \infty$. Moreover, we apply the {A}ubin--{L}ions compactness theorem \cite[Section 8, Corollary 4]{Sim87} and 
\eqref{tool3}
to obtain the
strong convergences (not re-labelled):
\begin{align*}
	\hat{v}_{h_k}^* \to v^* & \quad \mbox{strongly in } C\bigl([0,T];L^r (\Omega) \bigr),  \\
	\hat{v}_{\Gamma,h_k}^* \to v_\Gamma^* & \quad \mbox{strongly in } C\bigl([0,T];L^r (\Gamma) \bigr)
	\quad {\rm for~} r \in [2,\infty), \\
	\bar{\boldsymbol{v}}_{h_k}^*,
	\underline{\boldsymbol{v}}^*_{h_k}
	 \to \boldsymbol{v}^* & \quad \mbox{strongly in } L^2(0,T;\boldsymbol{H})
\end{align*}
as $k \to \infty$. According Lemma~\ref{linf}, we use 
$V \cap L^\infty (\Omega) \hookrightarrow \hookrightarrow L^r(\Omega) \subset H$ and
$V_\Gamma \hookrightarrow \hookrightarrow L^r(\Gamma) \subset H_\Gamma$ for all $r \in [2,\infty)$, 
where ``$\hookrightarrow \hookrightarrow$'' stands for the compact imbedding. 
From the demi-closedness of the maximal monotone operator $\boldsymbol{A}$,  
we obtain 
\begin{equation*}
	v^* = \gamma (v) \quad 
	\mbox{a.e.\ in } Q, 
	\quad 
	v_\Gamma^* =\gamma (v_\Gamma) 
	\quad 
	\mbox{a.e.\ on } \Sigma.
\end{equation*} 
Thus, from \eqref{ap1b}--\eqref{ap5b}, letting $k \to \infty$ in their weak formulation,
we see that $\boldsymbol{v}=(v,v_\Gamma)$ satisfy 
\begin{gather*}
	\int_\Omega \partial_t \gamma \bigl(v(t) \bigr) z dx 
	+ \int_\Gamma \partial_t \gamma \bigl(v_\Gamma(t) \bigr) z_\Gamma d\Gamma 
	+ \int_\Omega \nabla v(t) \cdot \nabla z dx  
	+ a \int_\Omega v(t) z dx  
	\nonumber \\
	{}
	+ \int_\Gamma \nabla_\Gamma v_\Gamma (t) \cdot \nabla_\Gamma z_\Gamma d\Gamma 
	+ b \int_\Gamma v_\Gamma (t)  z_\Gamma d\Gamma 
	= \int_\Omega f \bigl( \gamma \bigl(v(t) \bigr) \bigr) z dx 
	+ \int_\Gamma f_\Gamma \bigl( \gamma \bigl(v_\Gamma(t) \bigr) \bigr) z_\Gamma d\Gamma
\end{gather*}
for all $\boldsymbol{z}=(z,z_\Gamma) \in \boldsymbol{V}$, for a.a.\ $t \in (0,T)$, 
initial conditions $v(0)=v_0$ in $H$, and $v_\Gamma(0)=v_{\Gamma,0}$ in $H_\Gamma$. 
Let $z \in {\mathcal D}(\Omega)$, then $z_\Gamma =0$ and 
\begin{equation*}
	-\Delta v(t) =  f \bigl( \gamma \bigl(v(t) \bigr) \bigr) -a v(t) -\partial _t \gamma \bigl( v(t)\bigr) 
	\quad 
	\mbox{in } {\mathcal D}'(\Omega).
\end{equation*}
On the other hand, $f ( \gamma (v) )-\partial _t \gamma ( v) - a v \in L^2(0,T;H)$; therefore, we obtain 
$-\Delta v \in L^2(0,T;H)$ and 
\begin{equation}
	\partial _t \gamma \bigl(v(t) \bigr) -\Delta v(t) + a v(t)=  f \bigl( \gamma \bigl(v(t) \bigr) \bigr)
	\quad \mbox{in } H,
	\label{equ1}
\end{equation}
for a.a.\ $t \in (0,T)$. 
Next, for any $\boldsymbol{z} \in \boldsymbol{V}$, we see from \eqref{equ1} that 
\begin{gather}
	\int_\Gamma \partial_t \gamma \bigl(v_\Gamma(t) \bigr) z_\Gamma d\Gamma 
	+ \bigl\langle \partial _{\boldsymbol{\nu}} v(t), z_\Gamma \bigr\rangle
	+ \int_\Gamma \nabla_\Gamma v_\Gamma (t) \cdot \nabla_\Gamma z_\Gamma d\Gamma 
	+ b \int_\Gamma v_\Gamma (t)  z_\Gamma d\Gamma 
	\nonumber \\
	=
	\int_\Gamma f_\Gamma \bigl( \gamma \bigl(v_\Gamma(t) \bigr) \bigr) z_\Gamma d\Gamma,
	\label{weakg}
\end{gather}
for a.a.\ $t \in (0,T)$. 
Here, we apply the bootstrap argument for the dynamic boundary condition with surface diffusion 
to gain higher regularity (see, e.g., \cite{CC13, CF15, Fuk16, Fuk16b}).
We have already obtained $-\Delta v \in L^2(0,T;H)$ and 
$v_\Gamma \in L^2(0,T;V_\Gamma)$. 
Therefore, from the elliptic regularity theorem (see, e.g., \cite[Theorem~3.2, p.~1.79]{BG87}), we infer 
\begin{gather*}
	v \in L^2 \bigl (0,T;H^{3/2} (\Omega ) \bigr )
\end{gather*}
and consequently, from the trace theory with elliptic operator type \cite[Theorem~2.27, p.~1.64]{BG87}, we 
obtain 
$\partial_{\boldsymbol{\nu}} v \in L^2(0,T;H_\Gamma )$.
Therefore, from \eqref{weakg}, we also obtain 
$\Delta _\Gamma v_\Gamma \in L^2(0,T;H_\Gamma )$  
such that 
\begin{equation*} 
	\partial_t \gamma \bigl( v_\Gamma(t)\bigr) + \partial_{\boldsymbol{\nu}} v (t) 
	-\Delta _\Gamma v_\Gamma(t)+bv_\Gamma(t) = f_\Gamma \bigl( \gamma \bigl(v_\Gamma(t) \bigr) \bigr) 
	\quad \mbox{in } H_\Gamma,
	\label{equ2}
\end{equation*} 
for a.a.\ $t \in (0,T)$. 
Moreover, 
the information $-\Delta_{\Gamma } v_\Gamma \in L^2(0,T;H_\Gamma )$ 
implies 
$v_\Gamma \in L^2 ( 0,T;W_\Gamma)$ (see, e.g., \cite[p.~104]{Gri09}). 
Finally, this yields $v_\Gamma \in L^2 ( 0,T;H^{3/2}(\Gamma))$. 
Using the elliptic regularity theorem again, we see that 
$v \in L^2 (0,T;W)$ and $\boldsymbol{v}=(v,v_\Gamma)$ satisfies \eqref{af1}--\eqref{af5}. 
\smallskip

Next, we obtain the estimates \eqref{aenergy}--\eqref{aholder}. 
Firstly, the estimates \eqref{alinfty1}--\eqref{alinfty2} is the direct consequence of Lemma~3.1. 
Secondly, 
we obtain the uniform estimates 
for 
$\hat{\boldsymbol{v}}^{**}_h:=(\hat{v}^{**}_h,\hat{v}_{\Gamma,h}^{**})$
using\eqref{apone} \eqref{tool1}, for example
\begin{gather*}
	|\partial_t \hat{v}^{**}_h|_{L^2(0,T;H)}^2 \le M_1, \\
		\bigl| \hat{v}_h^{**}(t) \bigr|_{L^\infty(\Omega)} 
	\le e^{(\alpha+1)L T/(2\alpha)}\left( |v_0|_{L^\infty(\Omega)}+ |v_{\Gamma,0}|_{L^\infty(\Gamma)} \right)^{(\alpha+1)/2}
	\quad {\rm for~all~} t \in [0,T], \\
	|\hat{v}^{**}_h|_{L^2(0,T;H)}^2 \le \frac{h}{2} \bigl|v_0^{(\alpha+1)/2} \bigr|_H^2
	+|\bar{v}_h^{**}|_{L^2(0,T;H)}^2 \le M_6,
\end{gather*}
where $M_6$ is a positive constant independent of $n \in \mathbb{N}$. 
Then, there exists 
a subsequence (not re-labelled) and a
function
$\boldsymbol{v}^{**} =(v^{**},v_\Gamma^{**}) \in H^1(0,T;\boldsymbol{H}) \cap L^\infty(0,T;\boldsymbol{V} \cap \boldsymbol{L}^\infty )$ such that 
\begin{align*}
	\hat{\boldsymbol{v}}_{h_k}^{**} \to \boldsymbol{v}^{**} & \quad \mbox{weakly in } H^1(0,T;\boldsymbol{H}),  \\
	& \quad \mbox{weakly star in } L^\infty(0,T;\boldsymbol{V}), \\
	\hat{\boldsymbol{v}}_{h_k}^{**} \to \boldsymbol{v}^{**} & \quad \mbox{strongly in } C\bigl([0,T];\boldsymbol{H} \bigr),  \\
	\hat{v}_{h_k}^{**} \to v^{**} & \quad \mbox{a.e.\ in } Q, \\
	\hat{v}_{\Gamma, h_k}^{**} \to v^{**}_\Gamma & \quad \mbox{a.e.\ on }\Sigma
\end{align*}
as $k \to \infty$. Moreover, 
\begin{equation*}
	\boldsymbol{v}^{**}=\bigl( v^{(\alpha+1)/2},v_\Gamma^{(\alpha+1)/2} \bigr). 
\end{equation*}
\smallskip

Now, for all $t \in [0,T]$ and all $k \in \mathbb{N}$ with $h_k=T/n_k$,
there exists $i_k \in \{1,2,\ldots, n_k \}$ such that 
$t \in [(i_k-1) h_k, i_k h_k)$, if $t=T$ then put $i_k=n_k$.  
Moreover, $(i_k-1)h_k \to t$, $i_k h_k \to t$ as $k \to \infty$. Therefore, 
we rewrite \eqref{enelast} into the following form
\begin{align*}
	& 
	\frac{4\alpha}{(\alpha+1)^2} 
	\int_0^t	
	\left(
	\int_\Omega 
	\bigl| \partial _t \hat{v}_{h_k}^{**}(s)
	\bigr|^2 dx 
	+ 
	\int_\Gamma
	\bigl| \partial _t \hat{v}_{\Gamma, h_k}^{**}(s)
	\bigr|^2 d\Gamma
	\right) ds
	\nonumber \\
	 & \quad {}+ \varphi_1 \bigl ( \bar{\boldsymbol{v}}_{h_k} (t)\bigr) 
	 -
	  \int _\Omega \widehat{f} \bigl( \bar{v}_{h_k}^* (t)  \bigr) dx
	  -
	   \int _\Gamma \widehat{f}_\Gamma \bigl(\bar{v}_{\Gamma,h_k}^* (t)  \bigr) d\Gamma
	 \nonumber \\
	& \le 
	\frac{4\alpha}{(\alpha+1)^2} 
	\int_{(i_k-1) h_k }^t	
	\left(
	\int_\Omega 
	\bigl| \partial _t \hat{v}_{h_k}^{**}(s)
	\bigr|^2 dx 
	+ 
	\int_\Gamma
	\bigl| \partial _t \hat{v}_{\Gamma, h_k}^{**}(s)
	\bigr|^2 d\Gamma
	\right) ds
	\nonumber \\
	& \quad {}
	+
	\varphi_1 ( \boldsymbol{v}_{0}) 
	-
	\int _\Omega \widehat{f}_\gamma (v_{0} )  dx
	-
	\int _\Gamma \widehat{f}_{\Gamma,\gamma} (v_{\Gamma,0}) d\Gamma,
\end{align*}
and take the $\liminf_{k \to \infty}$ of both side, then, applying the lower semi-continuity and 
the {L}ebesgue dominated convergence theorem, we obtain the energy estimate \eqref{aenergy}. 
The $L^\infty$-boundedness \eqref{alinfty1}--\eqref{alinfty2} is a direct consequence of Lemma~\ref{linf}. 
Finally, using the fundamental inequality 
\begin{equation*}
	|r-s|
	\le 	
	\bigl| r^{(\alpha+1)/2}-s^{(\alpha+1)/2} \bigr|^{2/(\alpha+1)}
\end{equation*}
for all $r,s \ge 0$ (see Appendix), we obtain that 
\begin{align*}
	& 
	\bigl| v(t)-v(s) \bigr|_H^2 + \bigl| v_\Gamma (t)-v_\Gamma (s) \bigr|_H ^2
	\nonumber \\
	& \le 
	 \int_\Omega \bigl| v^{(\alpha+1)/2}(t)- v^{(\alpha+1)/2}(s) \bigr|^{4/(\alpha+1)} dx 
	+
	 \int_\Gamma \bigl| v_\Gamma^{(\alpha+1)/2}(t)- v_\Gamma^{(\alpha+1)/2}(s) \bigr|^{4/(\alpha+1)} d\Gamma
	\nonumber \\
	& \le \left( 
	|\Omega|^{(\alpha-1)/(\alpha+1)} 
	\bigl| \partial_t v^{(\alpha+1)/2} 
	\bigr|_{L^2(0,T;H)}^{4/(\alpha+1)}  
	+|\Gamma|^{(\alpha-1)/(\alpha+1)} 
	\bigl| \partial_t v_\Gamma^{(\alpha+1)/2} 
	\bigr|_{L^2(0,T;H_\Gamma)}^{4/(\alpha+1)}  
	\right)
	|t-s|^{2/(\alpha+1)}
\end{align*}
for all $s,t \in [0,T]$. 
Thus, we obtain the {H}\"older continuity \eqref{aholder}. 
\smallskip

The proof of uniqueness is quite standard. Let $\boldsymbol{w}:=(w,w_\Gamma)$ be 
the solution starting from the initial data $\boldsymbol{w}_0:=(w_0, w_{\Gamma,0})$ and 
compare it with $\boldsymbol{v}$.  
Define some approximation $\sigma_k \in C^1(\mathbb{R})$ of the signum function ${\rm sgn}$ satisfying 
$\sigma_k(0)=0$,
$-1\le \sigma_k(r) \le 1$, $\sigma_k'(r) \ge 0$, 
and $\sigma_k(r) \to {\rm sgn}r$ as $k \to \infty$, for all $r \in \mathbb{R}$. 
Taking the difference of the equations for $\boldsymbol{v}=(v,v_\Gamma)$ and $\boldsymbol{w}=(w,w_\Gamma)$ we obtain
\begin{gather*}
	\int_\Omega 
	\bigl( \partial_t \gamma (v) -\partial_t \gamma (w) \bigr) z dx 
	+ \int_\Gamma 
	\bigl( \partial_t \gamma (v_\Gamma)-\partial_t \gamma (w_\Gamma) \bigr) z_\Gamma d\Gamma 
	+ \int_\Omega \nabla (v-w) \cdot \nabla z dx  
	\nonumber \\
	+ a \int_\Omega (v-w) z dx  
	+ \int_\Gamma \nabla_\Gamma (v_\Gamma -w_\Gamma) \cdot \nabla_\Gamma z_\Gamma d\Gamma 
	+ b \int_\Gamma (v_\Gamma -w_\Gamma )  z_\Gamma d\Gamma 
	\nonumber \\
	= \int_\Omega \bigl(  f \bigl( \gamma (v) \bigr) - f\bigl( \gamma (w) \bigr) \bigr) z dx 
	+ \int_\Gamma \bigl( f_\Gamma \bigl( \gamma (v_\Gamma ) \bigr)-f_\Gamma \bigl( \gamma (w_\Gamma ) \bigr) \bigr) z_\Gamma d\Gamma
\end{gather*}
for all $\boldsymbol{z}=(z,z_\Gamma) \in \boldsymbol{V}$, for a.a.\ $t \in (0,T)$. 
Here, we omit the time variables $v=v(t)$ and $w=w(t)$. We take $z:=\sigma_k([v-w]^+)$ and 
$z_\Gamma :=\sigma _k([v_\Gamma-w_\Gamma]^+)$, where $[r]^+:=\max\{0, r \}$ for all $r \in \mathbb{R}$. 
Then, we have 
\begin{equation*}
	\int_\Omega \nabla (v-w) \cdot \nabla \sigma_k \bigl( [v-w]^+ \bigr) dx  
	= \int_\Omega \sigma_k' \bigl( [v-w]^+ \bigr) \bigl| \nabla [v-w]^+ \bigr|^2 dx  \ge 0
\end{equation*}
and the same kind of positivity for the term of $a(v-w)$, $\nabla_\Gamma(v_\Gamma-w_\Gamma)$, and $b(v_\Gamma-w_\Gamma)$, respectively. 
On the other hand, considering 
${\rm sgn} ([v-w]^+)={\rm sgn} ( [ \gamma(v)-\gamma(w) ]^+)$ and letting $k \to \infty$,
we obtain 
\begin{align*}
	& \int_\Omega \bigl( \partial_t \gamma (v) -\partial_t \gamma (w) \bigr)  \sigma_k \bigl( [v-w]^+ \bigr) dx  
	\nonumber \\
	& \to \int_\Omega \partial_t \bigl( \gamma (v) - \gamma (w) \bigr)   {\rm sgn} \bigl( [v-w]^+ \bigr) dx 
	\nonumber \\
	& = \int_\Omega \partial_t \bigl( \gamma (v) - \gamma (w) \bigr)  {\rm sgn} \bigl( \bigl[ \gamma(v)-\gamma(w) \big]^+ \bigr) dx 
	\nonumber \\
	& = \frac{d}{dt }\int_\Omega  \bigl| \bigl[ \gamma (v) -  \gamma (w) \bigr]^+ \bigr| dx
\end{align*}
and same as $( \partial_t \gamma (v_\Gamma) -\partial_t \gamma (w_\Gamma) )  \sigma_k ( [v_\Gamma-w_\Gamma]^+)$. 
Moreover, 
\begin{align*}
	& \int_\Omega \bigl(  f \bigl( \gamma (v) \bigr) - f\bigl( \gamma (w) \bigr) \bigr)  \sigma_k \bigl( [v-w]^+ \bigr) dx  
	\nonumber \\
	& \to \int_\Omega \bigl(  f \bigl( \gamma (v) \bigr) - f\bigl( \gamma (w) \bigr) \bigr)   {\rm sgn} \bigl( [v-w]^+ \bigr) dx 
	\nonumber \\
	& = \int_\Omega \bigl(  f \bigl( \gamma (v) \bigr) - f\bigl( \gamma (w) \bigr) \bigr)  {\rm sgn} \bigl( \bigl[ \gamma(v)-\gamma(w) \big]^+ \bigr) dx 
	\nonumber \\
	& \le L \int_\Omega  \bigl| \bigl[ \gamma (v) -  \gamma (w) \bigr]^+ \bigr| dx, 
\end{align*}
same as $(f_\Gamma( \gamma (v_\Gamma)) - f_\Gamma( \gamma (w_\Gamma) ) )  \sigma_k ( [v_\Gamma-w_\Gamma]^+ )$. 
By applying the {G}ronwall inequality 
 \begin{align*}
	& \bigl| \bigl[ \gamma \bigl( v(t) \bigr) -  \gamma \bigl(w(t) \bigr) \bigr]^+ \bigr|_{L^1(\Omega)} 
	 + \bigl| \bigl[ \gamma \bigl(v_\Gamma(t) \bigr) -  \gamma \bigl(w_\Gamma(t) \bigr) \bigr]^+ \bigr|_{L^1(\Gamma)}
	 \nonumber \\
	& \le \left(
	  \bigl| \bigl[ \gamma (v_0) -  \gamma (w_0) \bigr]^+ \bigr|_{L^1(\Omega)} 
	 + \bigl| \bigl[ \gamma (v_{\Gamma,0}) -  \gamma (w_{\Gamma,0}) \bigr]^+ \bigr|_{L^1(\Gamma)} 
	 \right)
	 e^{L t}
\end{align*}
for all $t \in [0,T]$, this comparison estimate gives us the uniqueness of the solution. 
\end{proof}

\section{Proof of main theorems}

For a convenience, we define 
\begin{align*}
	Y(\boldsymbol{z}) & :=\frac{1}{1+m}\int_\Omega z^{(1+m)/m} dx + \frac{1}{1+m}\int_\Gamma z_\Gamma^{(1+m)/m} d\Gamma \\
	& = \frac{\alpha}{\alpha+1}\int_\Omega z^{\alpha+1} dx + \frac{\alpha}{\alpha+1}\int_\Gamma z_\Gamma^{\alpha+1} d\Gamma.
\end{align*}
In the case of locally {L}ipschitz continuous perturbations, 
we apply the cut off method to prove local existence.

\begin{proposition}
\label{well}
Let $0<m \le 1$. Let us 
assume that 
$\boldsymbol{v}_0:=(v_0, v_{\Gamma,0}) \in \boldsymbol{V} \cap \boldsymbol{L}^\infty$ with 
$v_0 \ge 0$ and $v_{\Gamma,0} \ge 0$. Then there exist $T_{\rm max}>0$ depending on the initial data, 
as well as
a unique pair of non-negative functions 
$\boldsymbol{v}:=(v,v_\Gamma)$ 
such that they solve \eqref{f1}--\eqref{f5} on $[0, T_{\rm max})$. 
Moreover, the functions $v$ and $v_\Gamma$ belong to 
\begin{align*}
	& v \in C \bigl( [0,T_{\rm max}); L^{\alpha+1}(\Omega) \bigr) 
	\cap L^\infty \bigl(0,T;V \cap L^\infty(\Omega) \bigr) \cap L^2(0,T;W), \\
	& v^{(\alpha+1)/2} = v^{(1+m)/2m} \in H^1(0,T;H), \\
	& \gamma(v) \in H^1(0,T;H) \cap L^\infty (0,T;V),\\
	& v_\Gamma \in 
	C \bigl( [0,T_{\rm max}); L^{\alpha+1}(\Gamma) \bigr) 
	\cap L^\infty \bigl( 0,T;V_\Gamma \cap L^\infty (\Gamma)\bigr)  
	\cap L^2(0,T;W_\Gamma), \\
	& v_\Gamma^{(\alpha+1)/2} =v_\Gamma^{(1+m)/2m} \in H^1(0,T;H_\Gamma), \\
	& \gamma(v_\Gamma) \in H^1(0,T;H_\Gamma) \cap L^\infty (0,T;V_\Gamma)
\end{align*}
and satisfy the equality 
\begin{equation}
	\frac{d}{dt} Y \bigl( \boldsymbol{v}(t) \bigr) 
	+ 2 \varphi_1  \bigl( \boldsymbol{v}(t) \bigr) = 
	\lambda \int_\Omega v^{\alpha p +1}(t) dx 
	+	
	\mu \int_\Gamma v_{\Gamma}^{\alpha q +1}(t) d\Gamma
	\label{vol0}
\end{equation}
for a.a.\ $t \in (0,T_{\rm max})$. 
Furthermore, they satisfy the energy inequality
\begin{equation}
	\frac{4\alpha}{(\alpha+1)^2} 
	\int_s^t \left( 
	\int_\Omega \bigl| \partial_t \bigl(v^{(\alpha+1)/2}(\tau) \bigr)\bigr|^2 dx 
	+  \int_\Gamma \bigl| \partial_t \bigl( v_\Gamma^{(\alpha+1)/2}(\tau) \bigr) \bigr|^2 d\Gamma \right) d\tau 
	+ J \bigl(\boldsymbol{v}(t) \bigr) \le  J \bigl( \boldsymbol{v}(s) \bigr) 
	\label{energyst}
\end{equation}
for all $s, t \in [0,T_{\rm max})$ satisfying $s \le t$.
\end{proposition}

\begin{proof}
Let $M:=2(| v_0 |_{L^\infty(\Omega)} + | v_{\Gamma,0} |_{L^\infty(\Gamma)})$. 
Moreover, define $g_M, g_{\Gamma,M}:\mathbb{R} \to \mathbb{R}$ by 
\begin{gather*}
	g_M(r):=
	\begin{cases}
	-(M+1)^{\alpha p} &  \mbox{if } r<-(M+1)^{\alpha},\\
	|r|^{p-1}r &  \mbox{if } |r| \le  (M+1)^{\alpha}, \\
	(M+1)^{\alpha p} &  \mbox{if } r>(M+1)^{\alpha},
	\end{cases}
	\\
	g_{\Gamma, M}(r):=
	\begin{cases}
	-(M+1)^{\alpha q} &  \mbox{if } r<- (M+1)^{\alpha},\\
	|r|^{q-1}r &  \mbox{if } |r| \le  (M+1)^{\alpha}, \\
	(M+1)^{\alpha q} &  \mbox{if } r> (M+1)^{\alpha},
	\end{cases}
\end{gather*}
where $p,q > 1$. 
Then, applying Proposition~\ref{global}, for each $T>0$ there exists a unique  
$\boldsymbol{v}_M:=(v_M,v_{\Gamma,M})$ such that solves 
\eqref{af1}--\eqref{af5} with $f:=\lambda g_M$, $f_\Gamma:=\mu g_{M,\Gamma}$.  
Moreover, from the $L^\infty$-boundedness \eqref{alinfty1}--\eqref{alinfty2}, we have 
\begin{align*}
	\bigl| v_M(t) \bigr|_{L^\infty(\Omega)} 
	\le 
	M e^{p_* (M+1)^{\alpha (p_*-1)} t/\alpha }, \quad 
	\bigl| v_{\Gamma,M}(t) \bigr|_{L^\infty(\Gamma)} \le 
	M e^{p_*(M+1)^{\alpha(p_*-1)} t/ \alpha}
\end{align*}
for all $t \in [0,T]$, where we recall that $p_*=\lambda p+\mu q$. Now, taking $\delta >0$ satisfying 
\begin{equation}
	M e^{p_* (M+1)^{\alpha (p_*-1)} \delta/\alpha} \le M+1,
	\label{inv}
\end{equation}
$\boldsymbol{v}_M$ solves the original problem \eqref{f1}--\eqref{f5} on $[0,\delta]$. 
Moreover, we define 
\begin{equation*}
	T_{\rm max}:=\sup \bigl\{ \delta >0 \ : \ \mbox{the problem \eqref{f1}--\eqref{f5} has the unique solution on $[0,\delta]$} \bigr\}.
\end{equation*}
Thus, we have proved the local existence of \eqref{f1}--\eqref{f5} on $[0,T]$ for all $T \in (0,T_{\rm max})$. 
Additionally, from equations \eqref{f1}--\eqref{f5}, 
we obtain that $\boldsymbol{v}$ satisfies 
\eqref{energyst}. 
Next, from the characterization of the {S}obolev functions and the chain rule, we obtain that 
\begin{equation*} 
	\partial _t \gamma (v)=\partial_t v^\alpha = \partial _t \bigl( v^{(\alpha+1)/2} \bigr)^{2\alpha/(\alpha+1)} = 
	\frac{2\alpha}{\alpha+1} v^{(\alpha-1)/2} \partial_t v^{(\alpha+1)/2} 
\end{equation*}
and also, for $v_\Gamma^\alpha$ (see, \cite[Lemma~2.3]{Vit00}). Therefore, 
from the equations \eqref{f1}--\eqref{f3}, we obtain that the following equality holds:
\begin{equation}
	\int_s^t \left( \int_\Omega \partial_t \gamma \bigl(v(\tau) \bigr) v (\tau) dx 
	+ \int_\Gamma \partial_t \gamma \bigl( v_\Gamma(\tau) \bigr)v_\Gamma(\tau) d \Gamma \right) d\tau
	=Y \bigl( \boldsymbol{v}(t) \bigr) - Y \bigl( \boldsymbol{v} (s)\bigr)
	\label{equationst}
\end{equation}
for all $s,t \in [0, T_{\rm max})$, 
that is, $Y(\boldsymbol{v})$ is absolutely continuous on $[0,T_{\rm max})$. This implies \eqref{vol0} and 
additional continuities on $[0,T]$ in $L^{\alpha+1}(\Omega)$ and $L^{\alpha+1}(\Gamma)$.   
\end{proof}

Also, we obtain that 
$J(\boldsymbol{v}(t)) \le J(\boldsymbol{v}_0)$ for all $t \in [0,T_{\rm \max})$, and $J(\boldsymbol{v}(t))$ is 
monotone decreasing as $t \to T_{\rm max}$. Moreover, we obtain the 
invariance of ${\mathcal W}$ as follows.

\begin{lemma} The set 
${\mathcal W} \cap \boldsymbol{L}^\infty $ is invariant; that is, if 
$\boldsymbol{v}_0 \in {\mathcal W} \cap \boldsymbol{L}^\infty$, then 
$\boldsymbol{v}(t) \in {\mathcal W} \cap \boldsymbol{L}^\infty$ for all $t \in [0,T_{\max})$. 
\end{lemma}

The proof of this lemma is given in Subsection~4.2.

\subsection{Finite time extinction} 
Let 
$(a,b) \in\{(1,0), (0,1)\}$, 
$(\lambda, \mu)=(0,1)$. 
The strategy for proving Theorems \ref{ext1} and \ref{ext2} is the same as that in  
\cite[Theorem~2.1]{FF90} and \cite[Proposition~5]{Ota81}. 
\begin{proof} 
Assume that $\boldsymbol{v}_0 \in {\mathcal W} \cap \boldsymbol{L}^\infty$. 
Then, 
applying \cite[Lemma~2.5]{FF90} we can prove $T_{\max}=\infty$ (see, Remark~4.1). 
Recalling \eqref{vol0}, we have the following equality:
\begin{align}
	\frac{d}{dt} Y \bigl( \boldsymbol{v}(t) \bigr) 
	& = 
	-2 \varphi_1  \bigl( \boldsymbol{v}(t) \bigr) + (\alpha q +1 ) \varphi_2 \bigl( \boldsymbol{v}(t) \bigr) 
	\nonumber \\
	& = -2 \varphi_1  \bigl( \boldsymbol{v}(t) \bigr) + \int_\Gamma v_\Gamma^{\alpha q+1}(t) d\Gamma.
	\label{vol}
\end{align}
for a.a.\ $t \in (0,\infty)$. 
From \eqref{coercive} and the {S}obolev imbedding in $2$-dimensions, there exists a positive constant $C_{\rm S}>0$ such that 
\begin{align*}
	\left( \int_\Gamma z_{\Gamma}^{\alpha q +1} d\Gamma \right) ^{1/(\alpha q+1)}
	& =|z_\Gamma|_{L^{\alpha q+1}(\Gamma)} 
	\nonumber \\
	& \le C_{\rm S} |z_\Gamma|_{V_\Gamma}
	\nonumber \\
	& \le 2^{1/2}C_{\rm S} C_{\rm C}^{-1/2}
	\varphi_1 (\boldsymbol{z}) ^{1/2};
\end{align*}
that is, there exists a positive constant $C>0$ such that 
\begin{equation}
	\varphi_2 (\boldsymbol{z} ) \le C \varphi_1 (\boldsymbol{z} )^{(\alpha q +1)/2}.
	\label{case2}
\end{equation}
for all $\boldsymbol{z} \in \boldsymbol{V}$.
From the assumption, we have $J(\boldsymbol{v}(t)) \le J(\boldsymbol{v}_0):=d_0<d$ for all $t\in [0,\infty)$ and 
$2\varphi_1(\boldsymbol{v}_0) > (\alpha q +1) \varphi_2(\boldsymbol{v}_0)$. 
Moreover, $J(\boldsymbol{v}(t))$ is monotone decreasing.
Hence, there exists $\varepsilon_1 \in (0,1)$ such that 
\begin{equation} 
	(1-\varepsilon_1) 2 \varphi_1 \bigl( \boldsymbol{v}(t) \bigr) \ge (\alpha q+1) \varphi_2 \bigl( \boldsymbol{v}(t) \bigr) 
	\label{otani}
\end{equation}
for all $t \in [0,\infty)$ (see, Remark~4.2 and {\sc Figure~1}). 
Now, from equations \eqref{coercive}, \eqref{vol}, and \eqref{otani}
\begin{align*}
	0 & = 
	\frac{d}{dt} Y \bigl( \boldsymbol{v}(t) \bigr) 
	+2 \varphi_1  \bigl( \boldsymbol{v}(t) \bigr) - (\alpha q +1 ) \varphi_2 \bigl( \boldsymbol{v}(t) \bigr) 
	\nonumber \\
	& \ge 
	\frac{d}{dt} Y \bigl( \boldsymbol{v}(t) \bigr) 
	+2 \varepsilon_1  \varphi_1  \bigl( \boldsymbol{v}(t) \bigr)
	\nonumber \\
	& \ge 
	\frac{d}{dt} Y \bigl( \boldsymbol{v}(t) \bigr) 
	+\varepsilon_1 C_{\rm C} \bigl| \boldsymbol{v}(t) \bigr|_{\boldsymbol{V}}^2
\end{align*}
that is, under the assumption $1/5 < m$, from the {S}obolev imbedding in $3$-dimension, 
there exists a positive constant $C(\alpha)$ depending upon $\varepsilon_1$, $C_{\rm C}$, and $\alpha$, such that 
\begin{equation*}
	\frac{d}{dt} Y \bigl( \boldsymbol{v}(t) \bigr)  \le - C(\alpha) Y \bigl( \boldsymbol{v}(t) \bigr)^{2/(\alpha+1)}
\end{equation*}
for a.a.\ $t \in (0,\infty)$. Recalling the fact that $0 <2/(\alpha+1)<1$, we deduce 
\begin{equation*}
	Y \bigl( \boldsymbol{v}(t) \bigr)  
	\le 
	\left( 
	\left[ 
	Y ( \boldsymbol{v}_0 )^{(\alpha-1)/(\alpha+1)}
	- \frac{\alpha-1}{\alpha+1} C(\alpha) t \right]^+ \right) ^{(\alpha+1)/(\alpha-1)},
\end{equation*}
that is, there exists $T_{\rm ext}>0$ depending upon $Y(\boldsymbol{v}_0)$ such that 
$\boldsymbol{v}(t)\equiv \boldsymbol{0}$ for all $t \ge T_{\rm ext}$. 
Moreover, the estimate from above \eqref{upper} holds.  
\end{proof}

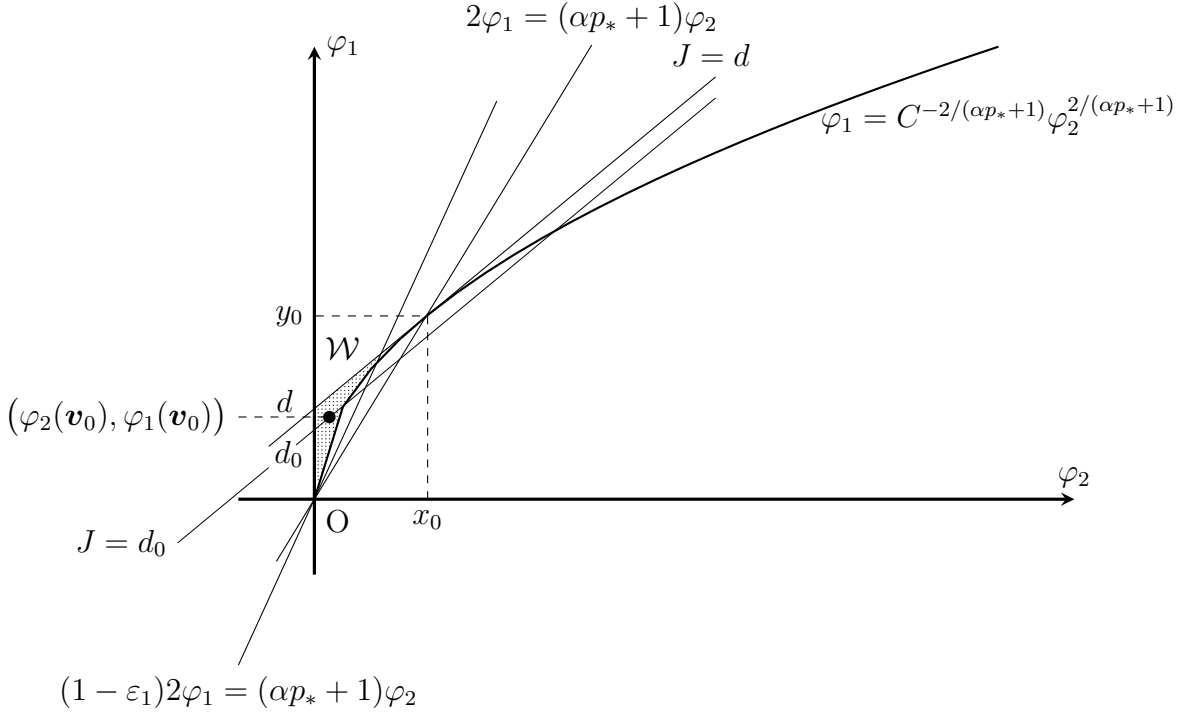
\begin{figure}[hbtp]
 \centering
 
\begin{tikzpicture}
 \draw[->,>=stealth,very thick] (-1,0)--(10,0)node[above]{$\varphi_2$}; 
 \draw[->,>=stealth,very thick] (0,-1)--(0,6)node[right]{$\varphi_1$}; 
 \draw (0,0)node[below right]{O}; 
 \draw[domain=-0.5:4.4] plot(1.2*\x,\x+1.2)node[above]{$J=d~~$};
 \draw[very thin, domain=-0.2:4.4] plot(1.2*\x,\x+0.92)node[right]{};
 \draw[very thin, domain=-0.5:-1.5] plot(1.2*\x,\x+0.92)node[left]{$J=d_0$};
 \draw[thick, domain=0:9] plot(\x,{2*pow(\x, 1/2)})node[below]{};
 \draw (9,5.5)node[below]{$\varphi_1=C^{-2/(\alpha p_*+1)} \varphi_2^{2/(\alpha p_*+1)}$};
 \draw (0,1.3)node[left]{$d~$};
 \draw (0,0.6)node[left]{$d_0$};

 \draw[dashed] (1.49,0)node[below]{$x_0$}--(1.49,2.43)--(0,2.43)node[left]{$y_0$}; 

 \draw (0,2)node[right]{$\mathcal W$};
 \draw[densely dotted] (0.05,1.20)--(0.05,0.25);
 \draw[densely dotted] (0.1,1.25)--(0.1,0.32);
 \draw[densely dotted] (0.15,1.3)--(0.15,0.5);
 \draw[densely dotted] (0.2,1.35)--(0.2,0.68);
 \draw[densely dotted] (0.25,1.4)--(0.25,0.84);
 \draw[densely dotted] (0.3,1.45)--(0.3,0.99);
 \draw[densely dotted] (0.35,1.45)--(0.35,1.12);
 \draw[densely dotted] (0.4,1.5)--(0.4,1.28);
 \draw[densely dotted] (0.45,1.55)--(0.45,1.33);
 \draw[densely dotted] (0.5,1.6)--(0.5,1.39);
 \draw[densely dotted] (0.55,1.65)--(0.55,1.48);
 \draw[densely dotted] (0.6,1.7)--(0.6,1.54);
 \draw[densely dotted] (0.65,1.7)--(0.65,1.62);
 \draw[densely dotted] (0.7,1.75)--(0.7,1.68);
 \draw[densely dotted] (0.75,1.8)--(0.75,1.72);
 \draw[densely dotted] (0.8,1.85)--(0.8,1.78);
 \draw[densely dotted] (0.85,1.9)--(0.85,1.82);
 \draw[densely dotted] (0.9,1.9)--(0.9,1.86);
 \draw[densely dotted] (0.95,2)--(0.95,1.94);
 \draw[domain=-0.5:3.65] plot(\x,1.65*\x)node[above]{$2\varphi_1=(\alpha p_*+1) \varphi_2$};
 \draw[domain=2.4:-1] plot(\x,2.2*\x)node[below]{$(1-\varepsilon_1) 2\varphi_1=(\alpha p_* +1) \varphi_2$};
 \draw(-0.05,1.08)node[right]{$\bullet$};
 \draw[dashed, domain=0.5:3](0.16,1.09)--(-1,1.09)node[left]{$\bigl(\varphi_2(\boldsymbol{v}_0), \varphi_1(\boldsymbol{v}_0) \bigr)$};
\end{tikzpicture}
\caption{Region of $\mathcal W$.}
\end{figure}

\begin{remark} 
We know that $\boldsymbol{v}(t)$ is the function in $\boldsymbol{L}^\infty$ for all $t \in [0,T_{\max})$ from the construction of the solution, 
see Proposition 4.1. 
However, 
to obtain the time global estimate with respect to $L^\infty$-norm, 
we do not apply Lemma~3.1 directly to $\boldsymbol{v}(t)$ 
any more since $g$ and $g_\Gamma$ are not global {L}ipschitz functions. 
To obtain it, we need the assumptions $1<p<5m$ or $q>1$. 
Indeed, we can apply the useful proposition with related to the {M}oser technique. 
Originally it was obtained by {A}likakos \cite[Lemma~3.2]{Ali79}, 
extended by {N}akao \cite[Lemma~3.1]{Nak85} for $m>1$, and {F}ila--{F}ilo 
\cite[Proposition~2.6]{FF90} for $0<m<1$. 
Under the assumption $1<p<5m$ or $q>1$, we  
can obtain the $L^\infty$-estimate independent of $t \in [0,T_{\rm max})$ \cite[Lemma~2.5]{FF90} and 
then we obtain that $T_{\rm max}=\infty$. 
Thus, we see that Theorem~2.2 is true also in the critical case $p =5m$ if we additionally assume that $T_{\rm max}=\infty$. 
\end{remark}
\smallskip

\begin{remark}
\label{remark}
The depth of the potential well $d$ of \eqref{d} is characterized by 
\begin{equation}
	d=\frac{\alpha p_* -1}{2} \left( \frac{2}{\alpha p_*+1}\right)^{(\alpha p_*+1)/(\alpha p_*-1)} C^{-2/(\alpha p_*-1)},
	\label{dd}
\end{equation}
where $C$ is the best constant satisfying \eqref{case2}. 
Let $(\varphi_2, \varphi_1)=(x_0,y_0)$ be the cross point between 
\begin{equation*}
	\begin{cases}
	\varphi_1=C^{-2/(\alpha p_*+1)} \varphi_2^{2/(\alpha p_*+1)}, \\
	2 \varphi_1 = (\alpha p_* + 1 ) \varphi_2.
	\end{cases}
\end{equation*}
Then, $J=d$ is the tangential line to the function 
$\varphi_1=C^{-2/(\alpha p_*+1)} \varphi_2^{2/(\alpha p_*+1)}$ at 
\begin{align*}
	& (x_0,y_0) \\
	& =\left( \left( \frac{2}{\alpha p_*+1} \right)^{(\alpha p_*+1)/(\alpha p_*-1)} C^{-2/(\alpha p_* -1)},  
	\frac{\alpha p_* +1}{2}  \left( \frac{2}{\alpha p_*+1} \right)^{(\alpha p_*+1)/(\alpha p_*-1)} C^{-2/(\alpha p_* -1)}\right)
\end{align*}
Thus, we can obtain \eqref{dd} from $d=y_0-x_0$.
\end{remark}

\subsection{Proof of invariance}

It remains to prove Lemma~4.1.
The proof of invariance is not difficult if we have 
$\boldsymbol{v} \in C([0,T_{\max}); \boldsymbol{V})$ (see \cite[Chapter~4]{FF89}, for example). 
However, in general, we do not have such regularity from the nonlinearity of $\gamma$. 
Therefore, we consider an approximation and obtain some information concerning the invariance: 

\begin{proof}
From the definition of $T_{\rm max}$, 
for each $\tau >0$ there exists $T_{\tau} \in (T_{\rm max}-\tau,T_{\rm max})$ such that 
the problem \eqref{f1}--\eqref{f5} has a unique solution on $[0,T_{\tau}]$. Moreover, 
recalling \eqref{inv} we put $M_\tau>0$ satisfying
\begin{equation*} 
	M e^{p_* (M+1)^{\alpha (p_*-1)} T_{\tau} /\alpha} \le M + M_\tau.
\end{equation*}
For each $\varepsilon>0$ and $0<T<\infty$, let us consider the following approximate problem
\begin{align}
	\partial _t \gamma(v_\varepsilon) + \varepsilon \partial_t v_\varepsilon -\Delta v_\varepsilon +a v_\varepsilon 
	=f \bigl( \gamma(v_\varepsilon) \bigr)
	& \quad \mbox{a.e.\ in } Q,
	\label{af1e}\\
	(v_\varepsilon)_{|_\Gamma}=v_{\Gamma, \varepsilon} 
	& \quad \mbox{a.e.\ on } \Sigma, 
	\label{af2e}\\
	\partial _t \gamma(v_{\Gamma, \varepsilon}) 
	+ \varepsilon \partial _t v_{\Gamma, \varepsilon}+ \partial _{\boldsymbol{\nu}} v_\varepsilon 
	- \Delta_\Gamma v_{\Gamma, \varepsilon} + b v_{\Gamma, \varepsilon} =f_{\Gamma} \bigl( \gamma(v_{\Gamma, \varepsilon} ) \bigr)
	& \quad \mbox{a.e.\ on } \Sigma,
	\label{af3e}\\
	v_\varepsilon(0)=v_0
	& \quad \mbox{a.e.\ in } \Omega, 
	\label{af4e}\\
	v_{\Gamma,\varepsilon}(0)=v_{\Gamma,0}
	& \quad \mbox{a.e.\ on }\Gamma,
	\label{af5e}
\end{align}
where $f:=\lambda g_{M+M_\tau}$ and $f:=\mu g_{\Gamma,M+M_\tau}$ are the same as in the proof of Proposition~4.1. 
Following Propositions~3.1 and 4.1 with the cut off technique, we prove that there exists a unique pair 
$\boldsymbol{v}_\varepsilon=(v_\varepsilon, v_{\Gamma, \varepsilon}) \in C([0,T]; \boldsymbol{V})$ of functions 
that satisfies \eqref{af1e}--\eqref{af5e} and 
\begin{align}
	& \frac{4\alpha}{(\alpha+1)^2} 
	\int_s^t 
	\left( 
	\bigl| \partial_t \bigl( \boldsymbol{v}_\varepsilon^{(\alpha+1)/2}(\tau) \bigr)\bigr|^2_{\boldsymbol{H}}
	+ \varepsilon \bigl| \partial_t \boldsymbol{v}_\varepsilon(\tau)
	\bigr|^2_{\boldsymbol{H}}\right)  d\tau 
		\nonumber \\
	 & \quad {}+ \varphi_1\bigl(\boldsymbol{v}_\varepsilon (t) \bigr) 
	 - \int _\Omega \widehat{f_\gamma } \bigl( v_\varepsilon (t) \bigr) dx
	 - \int _\Gamma \widehat{f}_{\Gamma,\gamma} \bigl( v_{\Gamma,\varepsilon} (t)  \bigr) d\Gamma
	 \nonumber \\
	& = \varphi_1\bigl(\boldsymbol{v}_\varepsilon (s) \bigr) 
	 - \int _\Omega \widehat{f_\gamma } \bigl( v_\varepsilon (s) \bigr) dx
	 - \int _\Gamma \widehat{f}_{\Gamma,\gamma} \bigl( v_{\Gamma,\varepsilon} (s)  \bigr) d\Gamma
	 \label{inv2}
\end{align}
for all $s, t \in [0,T]$ with $s \le t$. Moreover 
\begin{align*}
	\bigl| v_\varepsilon (t) \bigr|_{L^\infty(\Omega)} 
	& \le 
	M e^{p_* (M+1)^{\alpha (p_*-1)} t /\alpha } \\
	& \le M+M_\tau, 
	\\
	\bigl| v_{\Gamma,\varepsilon}(t) \bigr|_{L^\infty(\Gamma)} 
	& \le 
	M e^{p_*(M+1)^{\alpha(p_*-1)} t/ \alpha} \\
	& \le M+M_\tau
\end{align*}
for all $t \in [0,T_\tau]$. 
Therefore, we can replace 
$g_{M+M_\tau}$ by $g$ and $g_{\Gamma,M+M_\tau}$ by $g_\Gamma$ 
on the time interval $[0,T_\tau]$
in equations \eqref{af1e} and \eqref{af3e}. 
Furthermore, from \eqref{inv2} it can be shown that $J(\boldsymbol{v}_\varepsilon(t)) \le J (\boldsymbol{v}_0)$ and  
$\boldsymbol{v}_\varepsilon(t) \in {\mathcal W} \cap \boldsymbol{L}^\infty$ for all $t \in [0,T_\tau]$, 
since 
$\boldsymbol{v}_0 \in {\mathcal W} \cap \boldsymbol{L}^\infty$.
Moreover, 
we obtain the same kind of uniform estimates used in proving Propositions~\ref{global} and \ref{well}, such as there exists 
a subsequence $\{\varepsilon_k \}_{k \in \mathbb{N}}$ satisfying $\varepsilon_k \to 0$ as $k \to \infty$ such that 
\begin{align*}
	\boldsymbol{v}_{\varepsilon_k} 
	\to \boldsymbol{v}
	& \quad \mbox{weakly star in } L^\infty(0,T_\tau ;\boldsymbol{V}),  \\
	\gamma(\boldsymbol{v}_{\varepsilon_k}) 
	\to \gamma (\boldsymbol{v})  & \quad \mbox{weakly in } H^1(0,T_\tau;\boldsymbol{H}),  \\
	& \quad \mbox{weakly star in } L^\infty(0,T_\tau;\boldsymbol{V}),  \\
	\gamma(v_{\varepsilon_k}) \to \gamma(v) & \quad \mbox{strongly in } C\bigl([0,T_\tau];L^{r}(\Omega) \bigr), \\
    \gamma(v_{\Gamma, \varepsilon_k}) \to \gamma(v_{\Gamma}) & \quad \mbox{strongly in } C\bigl([0,T_\tau];L^{r}(\Gamma) \bigr) \quad {\rm for~} r \in [2,\infty), 
    \\
	\varepsilon_k \partial_t \boldsymbol{v}_{\varepsilon_k} \to \boldsymbol{0} & 
	\quad \mbox{strongly in } L^2(0,T_\tau;\boldsymbol{H})
\end{align*}
as $k \to \infty$, where 
$\boldsymbol{v}=(v,v_\Gamma)$ is the unique solution obtained in Proposition~4.1. 
Thus, we applied the compactness results \cite[Section 8, Corollary 4]{Sim87} again
to obtain the strong convergences
since, the compact imbeddings $V \cap L^\infty(\Omega) \hookrightarrow \hookrightarrow L^{r}(\Omega)$ and 
$V_\Gamma \hookrightarrow \hookrightarrow L^r(\Gamma)$ for $r \in [2, \infty)$ hold. 
Consequently, we deduce
\begin{equation}
	\varphi_2\bigl( \boldsymbol{v}_{\varepsilon_k}(t) \bigr)  \to 
	\varphi_2 \bigl( \boldsymbol{v}(t) \bigr)
	\label{limit}
\end{equation}
for all $t \in [0,T_{\rm max})$ 
as $k \to \infty$.
Recall the definition of ${\mathcal W}$.
We already know that $J(\boldsymbol{v}(t)) \le J(\boldsymbol{v}_0)<d$. 
Therefore, it is sufficient to prove that 
\begin{equation*}
	2 \varphi_1 \bigl( \boldsymbol{v}(t) \bigr) >(\alpha p_* +1) \varphi_2 \bigl( \boldsymbol{v}(t) \bigr)
\end{equation*}
for all $t \in [0,T_{\rm max})$. 
Let $t \in  [0,T_{\rm max})$. This is clear if $\varphi_2(\boldsymbol{v}(t))=0$, therefore we assume 
$\varphi_2(\boldsymbol{v}(t)) >0$. 
Now, from $d_0=J(\boldsymbol{v}_0)<d$, we obtain that $\delta_0:=d-d_0>0$ and 
\begin{equation*}
	J \bigl( \boldsymbol{v}(t)\bigr) \le d-\delta_0, \quad J \bigl( \boldsymbol{v}_{\varepsilon_k}(t)\bigr) \le d-\delta_0
\end{equation*}
for all $t \in [0,T_{\rm max})$. Thus, from \eqref{dd} we have 
\begin{align*}
	\varphi_1 \bigl( \boldsymbol{v}_{\varepsilon_k}(t) \bigr) -\varphi_2 \bigl( \boldsymbol{v}_{\varepsilon_k}(t) \bigr) 
	& \le d-\delta_0 \\
	& = \frac{\alpha p_* -1}{2} \left( \frac{2}{\alpha p_* +1}\right)^{(\alpha p_*+1)/(\alpha p_*-1)} C^{-2/(\alpha p_*-1)}- \delta_0.
\end{align*}
Now, $\boldsymbol{v}_{\varepsilon_k}(t) \in {\mathcal W}$, using \eqref{case2} we have 
\begin{align*}
	\frac{\alpha p_* - 1 }{2}\varphi_2 \bigl( \boldsymbol{v}_{\varepsilon_k}(t) \bigr) 
	& = \frac{\alpha p_* + 1 }{2}\varphi_2 \bigl( \boldsymbol{v}_{\varepsilon_k}(t) \bigr) 
	-\varphi_2 \bigl( \boldsymbol{v}_{\varepsilon_k}(t) \bigr)\\
	& < \varphi_1 \bigl( \boldsymbol{v}_{\varepsilon_k}(t) \bigr) -\varphi_2 \bigl( \boldsymbol{v}_{\varepsilon_k}(t) \bigr) \\
	& = \frac{\alpha p_* -1}{2} \left( \frac{2}{\alpha p_*+1}\right)^{(\alpha p_*+1)/(\alpha p_*-1)} C^{-2/(\alpha p_*-1)}- \delta_0 \\
	& \le  \frac{\alpha p_* -1}{2} \left( \frac{2}{\alpha p_*+1}\right)^{(\alpha p_*+1)/(\alpha p_*-1)} 
	\left( \frac{\varphi_1\bigl( \boldsymbol{v} (t) \bigr)^{(\alpha p_* + 1)/2}}{\varphi_2 \bigl(\boldsymbol{v} (t) \bigr)} \right)^{2/(\alpha p_*-1)}- \delta_0.
\end{align*}
Let $k \to \infty$ in the above; using \eqref{limit} we deduce that  
\begin{equation*}
	\varphi_2 \bigl( \boldsymbol{v}(t) \bigr) 
	\le \left( \frac{2}{\alpha p_*+1}\right)^{(\alpha p_*+1)/(\alpha p_*-1)} 
	\left( \frac{\varphi_1\bigl( \boldsymbol{v} (t) \bigr)^{(\alpha p_* + 1)/2}}{\varphi_2 \bigl(\boldsymbol{v} (t) \bigr)} \right)^{2/(\alpha p_*-1)}- \frac{2}{\alpha p_* - 1 }\delta_0,
\end{equation*}	
that is,
\begin{equation*}
	 (\alpha p_*+1)
	\varphi_2 \bigl( \boldsymbol{v}(t) \bigr)
	<
	2\varphi_1\bigl( \boldsymbol{v} (t) \bigr)
\end{equation*}
for all $t \in [0,T_{\rm max})$.
\end{proof}

Next, we prove Theorem \ref{ext2}.
Let 
$(a,b) \in\{(1,0), (0,1)\}$, $(\lambda, \mu)=(1,0)$.

\begin{proof}
Assume $1/5 < m <1$, $1 < p < 5m$. 
In this case,  the estimate \eqref{vol} is replaced by 
\begin{align*}
	\frac{d}{dt} Y \bigl( \boldsymbol{v}(t) \bigr) 
	& = 
	-2 \varphi_1  \bigl( \boldsymbol{v}(t) \bigr) + (\alpha p +1 ) \varphi_2 \bigl( \boldsymbol{v}(t) \bigr) 
	\nonumber \\
	& =
	-2 \varphi_1  \bigl( \boldsymbol{v}(t) \bigr) + \int_\Omega v^{\alpha p +1}(t) dx
\end{align*}
for a.a.\ $t \in (0,\infty)$. Now, there exists a positive constant $C_{\rm S}>0$ such that 
\begin{align*}
	\left( \int_\Omega z^{\alpha p +1} dx \right) ^{1/(\alpha p+1)}
	& =|z|_{L^{\alpha p+1}(\Omega)} 
	\nonumber \\
	& \le C_{\rm S} |z|_{V}
\end{align*}
for all $z \in V$ because $\alpha p+1 < 6$. Thus, there exists a positive constant $C>0$ such that 
\begin{equation*}
	\varphi_2 (\boldsymbol{z}) \le C \varphi_1 (\boldsymbol{z})^{(\alpha p +1)/2}.
\end{equation*}
Furthermore, 
there exists $\varepsilon_1 \in (0,1)$ such that 
\begin{equation*} 
	(1-\varepsilon_1) 2 \varphi_1 \bigl( \boldsymbol{v}(t) \bigr) \ge (\alpha p+1) \varphi_2 \bigl( \boldsymbol{v}(t) \bigr) 
\end{equation*}
for all $t \in [0,\infty)$ replaced with \eqref{otani}. 
Therefore, the proof is completely identical to the proof of 
Theorem~2.1. 
The proof of invariance of ${\mathcal W}$ is also the same; indeed, 
convergence \eqref{limit} holds since $\alpha p_*+1 = \alpha p+1 < 6$ from 
the additional assumption $1<p < 5m$. 
\end{proof}

\appendix
\section{}
\renewcommand{\theequation}{a.\arabic{equation}}
\setcounter{equation}{0}

We use the same settings as in the previous sections.

\begin{lemma} Let $\alpha \ge 1$. Then 
\label{fine}
\begin{gather*}
	\frac{4\alpha}{(\alpha+1)^2} 
	\bigl( r^{(\alpha+1)/2}-s^{(\alpha+1)/2} \bigr)^2 \le (r^\alpha-s^\alpha)(r-s), 
	\\
	|r^\alpha-s^\alpha|
	\le 	
	\frac{2\alpha}{\alpha+1} 
	\max\{ r,s \}^{(\alpha-1)/2}
	\bigl| r^{(\alpha+1)/2}-s^{(\alpha+1)/2} \bigr|,
	\\
	|r-s|
	\le 	
	\bigl| r^{(\alpha+1)/2}-s^{(\alpha+1)/2} \bigr|^{2/(\alpha+1)}
\end{gather*}
for all $r,s \ge 0$. 
\end{lemma}
\begin{proof} If $s=0$,  then we can prove that all inequalities hold.
Therefore, it is sufficient to prove that
\begin{gather*}
	F(x) := (x^\alpha-1)(x-1)-\frac{4\alpha}{(\alpha+1)^2} \bigl( x^{(\alpha+1)/2}-1 \bigr)^2, \\
	G(x):=\frac{2\alpha}{\alpha+1}x^{(\alpha-1)/2}\bigl (x^{(\alpha+1)/2}-1 \bigr)-(x^\alpha-1), \\
	H(x):=\bigl(x^{(\alpha+1)/2}-1 \bigr)-(x-1)^{(\alpha+1)/2}
\end{gather*}
are positive for $x \ge 1$. 
Firstly, from the basic calculation, we see that $G(1)=0$ and
\begin{equation*}
	G'(x) =\frac{\alpha(\alpha-1)}{\alpha+1} x^{(\alpha-3)/2} \bigl( x^{(\alpha+1)/2} -1\bigr) 
	\ge 0, 
\end{equation*}
these imply that $G(x) \ge 0$ for all $x \ge 1$.
Secondly, we put $\ell:=(\alpha+1)/2$; then
\begin{align*}
	F(x) & = ( x^{2\ell-1}-1 )(x-1)-\frac{2\ell-1}{\ell^2} (x^\ell-1)^2 \\
	& = \left( \frac{\ell-1}{\ell}\right)^2 (x^\ell-1)^2-x(x^{\ell-1}-1)^2 \\
	& = \left( \frac{\ell-1}{\ell} (x^\ell-1)+x^{1/2}(x^{\ell-1}-1) \right)  \left( \frac{\ell-1}{\ell} (x^\ell-1)-x^{1/2}(x^{\ell-1}-1) \right) \\
	& =:\left( \frac{\ell-1}{\ell} (x^\ell-1)+x^{1/2}(x^{\ell-1}-1) \right) F_1(x),
\end{align*}
where the multiplier of $F_1$ is positive. Moreover, we 
can prove that $F_1(x) \ge 0$ for all $x \ge 1$ just as we could for $G(x)$. 
This means that $F(x) \ge 0$ for all $x \ge 1$. 
Finally, $H(1)=0$ and 
\begin{equation*}
	H'(x) = \ell x^{\ell-1}-\ell (x-1)^{\ell -1} 
	\ge 0
\end{equation*}
for all $x \ge 1$ since $\ell x^{\ell-1}$ is monotone increasing for $x \ge 1$. 
This implies that $H(x) \ge 0$ for all $x \ge 1$. As a remark, a similar 
inequality for $H(x)$ can be obtained from the Tartar inequality.  
\end{proof}


\end{document}